\documentclass[dvips,ps]{imsart}

\RequirePackage[OT1]{fontenc}
\RequirePackage{amsthm,amsmath}
\RequirePackage[numbers]{natbib}
\RequirePackage[colorlinks,citecolor=blue,urlcolor=blue]{hyperref}

\usepackage{amsfonts,amssymb,amstext}
\usepackage{amsmath}
\usepackage{amstext}
\usepackage{amsopn}

\pubyear{2013}
\volume{0}
\issue{0}
\startlocaldefs
\numberwithin{equation}{section}
\theoremstyle{plain}
\newtheorem{thm}{Theorem}[section]
\newtheorem{lem}{Lemma}[section]
\newtheorem{pp}{Proposition}[section]

\newtheorem{exa}{Example}[section]
\newtheorem{rem}{Remark}[section]
\newtheorem{den}{Definition}[section]

\newcommand{\tr}{\hbox{tr}}

\newcommand{\beq}{\begin{equation}}
\newcommand{\eeq}{\end{equation}}
\newcommand{\beqq}{\begin{equation*}}
\newcommand{\eeqq}{\end{equation*}}
\newcommand{\bea}{\begin{eqnarray}}
\newcommand{\eea}{\end{eqnarray}}
\newcommand{\beaa}{\begin{eqnarray*}}
\newcommand{\eeaa}{\end{eqnarray*}}

\newcommand{\n}{\noindent}
\newcommand{\q}{\quad}
\newcommand{\qq}{\qquad}

\newcommand{\Ph}{{\rm I\!P}}

\newcommand{\g}{\gamma}
\newcommand{\I}{\varphi}
\newcommand{\G}{\Gamma}
\newcommand{\de}{\delta}
\newcommand{\De}{\Delta}

\newcommand{\al}{\alpha}
\newcommand{\la}{\lambda}
\newcommand{\f}{\infty}
\newcommand{\vs}{\varepsilon}
\newcommand{\cd}{\cdot}
\newcommand{\si}{\sigma}

\newcommand{\be}{\beta}

\newcommand{\Om}{\Omega}
\newcommand{\om}{\omega}
\newcommand{\itf}{\int_{-\infty}^{+\infty}}

\newcommand{\sm}{\setminus}

\newcommand{\inl}{\int_{-\pi}^{\pi}}

\newcommand {\ol} {\overline}

\newcommand {\uu}{{\bf u}}
\newcommand {\0}{{\bf 0}}
\newcommand {\rr}{{\mathbb R}}

\newcommand {\s} {\section}

\newcommand{\Lip}{{\rm Lip}}

\newcommand{\an}{\subset\!\!\!\Longrightarrow}

\endlocaldefs

\begin{document}

\begin{frontmatter}

\title{The Trace Problem for Toeplitz Matrices and
 Operators and its Impact in Probability\thanksref{T1}}
\runtitle{The Trace Problem and its Impact in Probability}
\thankstext{T1}{This is an original survey paper.}

\begin{aug}
\author{\fnms{Mamikon S.} \snm{Ginovyan}\thanksref{t1,t2}
\ead[label=e1]{ginovyan@math.bu.edu}}

\address{Department of Mathematics and Statistics\\
Boston University, Boston, MA, USA\\
\printead{e1}}

\author{\fnms{Artur A.} \snm{Sahakyan}%\thanksref{t3}
\ead[label=e2]{sart@ysu.am}}

\address{Department of Mathematics and Mechanics\\
Yerevan State University, Yerevan, Armenia\\
\printead{e2}}

\and
\author{\fnms{Murad S.} \snm{Taqqu}\thanksref{t3}
\ead[label=e3]{murad@math.bu.edu}}

\address{Department of Mathematics and Statistics\\
Boston University, Boston, MA, USA\\
\printead{e3}}

\thankstext{t1}{Corresponding author}
\thankstext{t2}{Supported in part by NSF Grant \#DMS-0706786}
\thankstext{t3}{Supported in part by NSF Grants \#DMS-0706786 and \#DMS-1007616}
\runauthor{M. Ginovyan et al.}

\affiliation{Some University and Another University}

\end{aug}

\begin{abstract}
The trace approximation problem for Toeplitz matrices and its
applications to stationary processes dates back to the classic book
by Grenander and Szeg\"o, {\em Toeplitz forms and their applications}.
It has then been extensively studied in the literature.

In this paper we provide a survey and unified treatment of the
trace approximation problem both for Toeplitz matrices and for operators
and describe applications to discrete- and continuous-time
stationary processes.

The trace approximation problem 
serves indeed as a tool to study many probabilistic and statistical
topics for stationary models. These include central and non-central
limit theorems and large deviations of Toeplitz type random
quadratic functionals, parametric and nonparametric estimation,
prediction of the future value based on the observed past of the
process, etc.

We review and summarize the known results concerning the trace
approximation problem, prove some new results, and provide
a number of applications to discrete- and continuous-time stationary
time series models with various types of memory structures,
such as long memory, anti-persistent and short memory.

\end{abstract}

\begin{keyword}[class=AMS]
\kwd[Primary ]{60G10}
\kwd{62G20}
\kwd[; secondary ]{47B35}
\kwd{15B05}
\end{keyword}

\begin{keyword}
\kwd{Stationary process}
\kwd{Spectral density}
\kwd{Long-memory}
\kwd{Central limit theorem}
\kwd{Toeplitz operator}
\kwd{Trace approximation}
\kwd{Singularity}
\end{keyword}

\tableofcontents
\end{frontmatter}

\section{Introduction}
\label{sec:1}

Toeplitz matrices and operators, which have great independent
interest and a wide range of applications in different fields of
science (economics, engineering, finance, hydrology, physics,
signal processing, etc.), arise naturally in the context of
stationary processes. This is because
the covariance matrix (resp. operator) of a discrete-time (resp.
continuous-time) stationary process is a truncated Toeplitz matrix
(resp. operator) generated by the spectral density of that process,
and vice versa, any non-negative summable function generates a
Toeplitz matrix (resp. operator), which can be considered as a spectral
density of some discrete-time (resp. continuous-time) stationary
process, and therefore the corresponding truncated Toeplitz matrix
(resp. operator) will be the covariance matrix (resp. operator)
of that process.

Truncated Toeplitz matrices and operators are of particular importance,
and serve as tools, to study many topics in the spectral and statistical
analysis of discrete- and continuous-time stationary processes, such as
central and non-central limit theorems and large deviations of Toeplitz type
random quadratic forms and functionals, estimation of the spectral
parameters and functionals, asymptotic expansions of the estimators,
hypotheses testing about the spectrum, prediction of the future
value based on the observed past of the process, etc.
(see, e.g., \cite{Ad} - \cite{Ber}, \cite{BP1,BD,BrD,DC,Dah1},
\cite{FT1} - \cite{TT3},
and references therein).

The present paper is devoted to the problem of approximation of
the traces of products of truncated Toeplitz matrices and operators
generated by integrable real symmetric functions defined on the
unit circle (resp. on the real line).
We discuss estimation of the corresponding errors, and describe
applications to discrete- and continuous-time stationary time
series models with various types of memory structures
(long-memory, anti-persistent and short-memory).

The paper contains a number of new theorems both for Toeplitz matrices
and operators, which are stated in Sections \ref{ATM} and \ref{ATO},
respectively. Those in Section \ref{ATO} involving Toeplitz operators
are proved in Section \ref{PR}.
The corresponding theorems of Section \ref{ATM} involving Toeplitz
matrices can be proved in a similar way and hence are omitted.
Section \ref{ASP} concerns applications and also contains some
new results. These are proved within the section.

The paper is organized as follows.
In the remainder of this section we state the trace approximation
problem and describe the statistical model.
In Section \ref{ATM} we discuss the trace problem for Toeplitz matrices.
Section \ref{ATO} considers the same problem for Toeplitz operators.
Section \ref{ASP} is devoted to some applications of the trace problem
to discrete- and continuous-time stationary processes.
In Section \ref{PR} we outline the proofs of Theorems \ref{T5} - \ref{T5-3}.
An appendix contains the proofs of technical lemmas.

\subsection{The Trace Approximation Problem}
\label{Tr-Problem}
The problem of approximating traces of products of truncated
Toeplitz matrices and operators
can be stated as follows.

Let $\mathcal{H}=\{h_1,h_2,\ldots,h_m\}$ be a collection of integrable
real symmetric functions defined on the domain $\Lambda$, where
$\Lambda=\mathbb{R}:=(-\f, \f)$ or $\Lambda=\mathbb{T}:= (-\pi.\pi]$,
and let $A_T(h_k)$ denote either the $T$-truncated Toeplitz operator,
or the $T\times T$ Toeplitz matrix generated by function $h_k$
(the corresponding definitions are given below).
Further, let
\beq
\label{lf-1}
\tau:=\{\tau_k: \, \tau_k\in \{-1, 1\}, \,\, k=\overline{1, m}\}
\eeq
be a given sequence of $\pm 1$'s. Define
\beq
\label{im-1}
S_{A,\mathcal{H},\tau}(T):=
\frac1T\,\tr\left[\prod_{k=1}^m\{A_T(h_k\}^{\tau_k}\right],
\eeq
where $\tr[A]$ stands for the trace of $A$,
\beq
\label{im-2}
M_{\Lambda, \mathcal{H},\tau}:=
(2\pi)^{m-1}\int_\Lambda\prod_{k=1}^m [h_k(\la)]^{\tau_k}d\la,
\eeq
and
\beq
\label{im-3}
\De_{A,\Lambda,\mathcal{H},\tau}(T):
=|S_{A,\mathcal{H},\tau}(T)-M_{\Lambda, \mathcal{H},\tau}|.
\eeq
The problem is to approximate $S_{A,\mathcal{H},\tau}(T)$ by
$M_{\Lambda, \mathcal{H},\tau}$ and estimate the error rate for
$\De_{A,\Lambda,\mathcal{H},\tau}(T)$ as $T\to\f$.
More precisely, for a given sequence
$\tau=\{\tau_k\in \{-1, 1\}, \, k=\overline{1, m}\}$
find conditions on functions $\{h_k(\la), \,k=\overline{1, m}\}$
such that:
\beaa
&&{\rm Problem \,\, (A):} \q \De_{A,\Lambda,\mathcal{H},\tau}(T)=o(1)
\q {\rm as} \q T\to\f, \q{\rm or}\\
&&{\rm Problem \,\,  (B):} \, \q \De_{A,\Lambda,\mathcal{H},\tau}(T)=O(T^{-\g}),
\q \g>0, \q {\rm as} \q T\to\f.
\eeaa

\n The trace approximation problem goes back to the classical
monograph by Grenander and Szeg\"o \cite{GS}, and has been
extensively studied in the literature
(see, e.g., Kac \cite{Kc}, Rosenblatt \cite{R2}, \cite{R3},
Ibragimov \cite{I},
Taniguchi \cite{Ta}, Avram \cite{A}, Fox and Taqqu \cite{FT1}, Taqqu
\cite{Tq2},  Dahlhaus \cite{Dah1}, Giraitis and Surgailis \cite{GSu},
Ginovyan \cite{G0}, Taniguchi and Kakizawa \cite{TK}, Lieberman and
Phillips \cite{LP}, Giraitis et al. \cite{GKSu},
Ginovyan and Sahakyan \cite{GS1}-\cite{GS6}, and references therein).

In this paper we review and summarize the known results concerning
Problems (A) and (B), prove some new results, as well as provide
a number of applications to discrete- and continuous-time
stationary time series models that have various types of
memory structures (short-, intermediate-, and long-memory).

We focus on the following special case
which is important from an application viewpoint,
and is commonly discussed in the literature:
$m=2\nu$, $\tau_k=1, \, k=\overline{1, m}$
(or $\tau_k=(-1)^k, \, k=\overline{1, m}$), and
\bea
\label{rim-1}
&&h_1(\la)= h_3(\la)=\cdots =h_{2\nu-1}(\la):=f(\la)\\
\label{rim-2}
\nonumber
&&h_2(\la)= h_4(\la)=\cdots =h_{2\nu}(\la):=g(\la).
\eea

Throughout the paper the letters $C$ and $c$,
with or without index, are used to denote positive constants,
the values of which can vary from line to line.
Also, all functions defined on $\mathbb{T}$ are assumed to be
$2\pi$-periodic and periodically extended to $\mathbb{R}$.

\subsection{The Model: Short, Intermediate and Long Memory Processes.}
\label{model}
\n
Let $\{X(u), \ u\in \mathbb{U}\}$ be a centered, real-valued,
continuous-time or discrete-time second-order
stationary process with covariance function $r(u)$,
possessing a spectral density function
$f(\la),$ $\la\in \Lambda$, that is,
$$E[X(u)]=0, \q
r(u)=E[X(t+u)X(t)], \q u, t\in \mathbb{U},$$
and $r(u)$ and $f(\la)$ are connected by the Fourier integral:
\beq
\label{mo1}
r(u)=\int_\Lambda\exp\{i\la u\}f(\la)d\la, \q u\in \mathbb{U}.
\eeq
Thus, the covariance function $r(u)$ and the spectral density
function $f(\la)$ are equivalent specifications of second order
properties for a stationary process $X(u)$.

\n
The time domain $\mathbb{U}$ is the real line $\mathbb{R}$
in the continuous-time case, and the set of integers
$\mathbb{Z}$ in the discrete-time case.
The frequency domain $\Lambda$ is $\mathbb{R}$
in the continuous-time case, and $\Lambda=\mathbb{T}=(-\pi, \pi]$
in the discrete-time case.
In the continuous-time case the process $X(u)$ is also assumed measurable
and mean-square continuous:
$\mathbb{E}[X(t)-X(s)]^2\to 0$ as $t\to s$.

The statistical and spectral analysis of stationary processes requires
{\em two types of conditions\/} on the spectral density $f(\la).$
The first type controls the {\em singularities} of $f(\la)$,
and involves the {\em dependence (or memory) structure } of the
process, while the second type - controls the {\em smoothness} of $f(\la).$

\vskip 2mm

{\bf Dependence (memory) structure of the model}.
We will distinguish the following types of stationary models:

(a) short memory (or short-range dependent),

(b) long memory (or long-range dependent),

(c) intermediate memory (or anti-persistent).

\n
The memory structure of a stationary process is essentially a
measure of the dependence between all the variables in the process,
considering the effect of all correlations simultaneously.
Traditionally memory structure has been defined in the time domain
in terms of decay rates of long-lag autocorrelations, or in the
frequency domain in terms of rates of explosion of low frequency spectra
(see, e.g., Beran \cite{Ber}, Guegan \cite{Gu},
Robinson \cite{Ro2}, and references therein).

It is convenient to characterize the memory structure in terms
of the spectral density function.

\vskip 2mm
{\bf Short-memory models}.
Much of statistical inference
is concerned with {\sl short-memory} stationary models,
where the spectral density  $f(\lambda)$ of the
model is bounded away from zero and infinity,
that is, there are constants  $C_1$ and $C_2$ such that
\beqq
0< C_1 \le f(\la) \le C_2 <\f.
\eeqq

A typical short memory model example is the stationary
Autoregressive Moving Average (ARMA)$(p,q)$ process
whose covariance $r(u)$ is exponentially bounded:
$$|r(k)|\le Cr^{-k}, \q k=1,2,\ldots;\q 0<C<\f; \,\, 0<r<1,$$
and the spectral density is a rational function.

\vskip 2mm
{\bf Discrete-time long-memory and anti-persistent models}.
Data in many fields of science (economics, finance,
hydrology, etc.), however, is well modeled by a stationary process with
{\sl unbounded} or {\sl vanishing} spectral density
(see, e.g., Beran \cite{Ber}, Guegan \cite{Gu}, Palma \cite{Pal},
Taqqu \cite{Tq1} and references therein).

A {\sl long-memory} model is defined to be a
stationary process with {\sl unbounded} spectral density,
and  an {\sl anti-persistent} model -- a stationary
process with {\sl vanishing} spectral density.

In the discrete context, a basic long-memory model
is the Autoregressive Fractionally Integrated Moving Average
(ARFIMA)$(0,d,0))$ process $X(t)$ defined by
\beqq
 (1-B)^dX(t)=\varepsilon(t), \q 0<d<1/2,
 \eeqq
where $B$ is the backshift operator $BX(t)=X(t-1)$ and
$\varepsilon(t)$ is a discrete-time white noise.
The spectral density of $X(t)$ is given by
\beq
\label{MT0}
 f(\la)=|1-e^{i\la}|^{-2d}=(2\sin(\la/2))^{-2d},
 \q 0<\la\le\pi, \q 0<d<1/2.
\eeq
A typical example of an {\sl anti-persistent} model is the
ARFIMA$(0,d,0)$ process $X(t)$ with spectral density
$f(\la)=|1-e^{i\la}|^{-2d}$ with $d<0$.

Note that the condition $d<1/2$ ensures that
$\int_{-\pi}^\pi f(\la)d\la<\f$,
implying that the process $X(t)$ is well defined because
$E[|X(t)|^2]=\int_{-\pi}^\pi f(\la)d\la.$

Data can also occur in the form of a realization of a "mixed"
short-long-intermediate-memory stationary process $X(t)$
with spectral density
$$f(\la)=f_I(\la)f_L(\la)f_S(\la),$$
where $f_I(\la)$, $f_L(\la)$ and $f_S(\la)$ are the intermediate,
long- and short-memory components, respectively.
A well-known example of such process, which appears in many applied
problems, is an ARFIMA$(p,d,q)$ process with spectral density
\beq
\label{AA}
f(\la)=|1-e^{i\la}|^{-2d}h(\la), \q d<1/2,
\eeq
where $h(\la)$ is the spectral density of an ARMA$(p,q)$ process.
\n
Observe that for $ 0<d<1/2$ the model $X(t)$ specified by (\ref{AA})
displays long-memory, for $d<0$ - intermediate-memory,
and for $d=0$ - short-memory.
For $d \ge1/2$ the function $f(\la)$ is not integrable,
and thus it cannot represent a spectral density of a stationary
process. Also, if $d \le-1$, then the series $X(t)$ is not
invertible in a sense that it cannot be used to recover a white noise
$\varepsilon(t)$ by passing $X(t)$ through a linear filter (see \cite{BP1,BD}).

The ARFIMA$(p,d,q)$ processes, first introduced by Granger and Joyeux
\cite{GJ}, and Hosking \cite{Hos}, became very popular due to their ability in
providing a good characterization of the long-run properties of many
economic and financial time series. They are also very useful for
modeling multivariate time series, since they are able to capture a
larger number of long term equilibrium relations among economic
variables than the traditional multivariate ARIMA models
(see, e.g., Henry and Zaffaroni \cite{HZ} for a survey on this topic).

Another important long-memory model is the fractional Gaussian noise (fGn).
To define fGn first consider the {\it fractional Brownian motion\/} (fBm)
$\{B(t):=B_H(t), t\in\mathbb{R}\}$ with Hurst index $H$, $0<H<1$,
defined to be a centered Gaussian $H$-self-similar process having stationary
increments, that is, $B_H(t)$ satisfies the following conditions:

(a) $\mathbb{E}[B(t)]=0$;

(b) $B(at)\stackrel{d}=a^HB(t)$ for any $a>0$ and $t\in\mathbb{R}$;

(c) $B(t+u)-B(u)\stackrel{d}=B(t)-B(0)$ for any $t,u\in\mathbb{R}$;

(d) the covariance function is given by
$${\rm Cov}(B(s),B(t))=\frac{\sigma_0^2}2\left[|t|^{2H}-|s|^{2H}-|t-s|^{2H}\right],$$
where $\sigma_0^2={\rm Var} B(1)$, and where the symbol $\stackrel{d}=$ stands for equality of the finite-dimensional distributions.
Then the increment process
$$\{X(k):= B_H(k+1)-B_H(k), k\in\mathbb{Z}\},$$
called {\it fractional Gaussian noise\/} (fGn), is a discrete-time centered Gaussian stationary process with covariance function
\beq
\label{MT1}
r(t))=\frac{\sigma_0^2}2\left[|k+1|^{2H}-|k|^{2H}-|k-1|^{2H}\right],
\q k\in\mathbb{Z}
\eeq
and spectral density function
\beq
\label{MT2}
f(\la)=c\, |1-e^{i\la}|^{2}\sum_{k=-\f}^\f|\la+2\pi k|^{-(2H+1)},
\q -\pi\le\la\le\pi,
\eeq
where $c$ is a positive constant.

It follows from (\ref{MT2}) that $f(\la)\thicksim c\, |\la|^{1-2H}$
as $\la\to 0$, that is, $f(\la)$ blows up if $H>1/2$
and tends to zero if $H<1/2$.
Also, comparing (\ref{MT0}) and (\ref{MT2}), we observe that,
up to a constant, the spectral density of fGn has the same
behavior at the origin as ARFIMA$(0,d,0)$ with $d=H-1/2.$

Thus, the fGn $\{X(k), k\in\mathbb{Z}\}$ has long-memory
if $1/2<H<1$ and is anti-percipient if $0<H<1/2$.
The variables  $X(k)$'s are independent if $H=1/2$.
For more details we refer to Samorodnisky and Taqqu \cite{ST}
and Taqqu \cite{Tq1}.

\vskip 2mm

{\bf Continuous-time long- memory and anti-persistent models}.
In the continuous context, a basic process which has
commonly been used to model long-range dependence is fractional
Brownian motion (fBm) $B_H$ with Hurst index $H$, defined above.
It can be regarded as Gaussian process having a spectral density:
%of the form
\beq
\label{lr3}
f(\la)=c|\la|^{-(2H+1)}, \q c>0, \ \,\,
 0<H<1, \  \,\, \la\in\mathbb{R},
 \eeq
where (\ref{lr3}) can be understood in a generalized sense
since the fBm $B_H$ is a nonstationary process
(see, e.g.,  Anh et al. \cite{ALM}
and Gao et al. \cite{GAHT}).

A proper stationary model in lieu of fBm is the
{\it fractional Riesz-Bessel motion} (fRBm),
introduced in Anh et al. \cite{AAR}, and defined as a
continuous-time Gaussian process $X(t)$ with
spectral density
\beq
\label{rb1}
f(\la)=c\,|\la|^{-2\al}(1+\la^2)^{-\be},\q \la\in\mathbb{R},
\, 0<c<\f, \, 0<\al<1, \,\be>0.
\eeq

\n
The exponent $\al$ determines the long-range dependence,
while the exponent $\be$
indicates the second-order intermittency of the process
(see, e.g., \cite{ALM,GAHT}).

Notice that the process $X(t)$, specified by (\ref{rb1}),
is stationary if $0<\al<1/2$
and is non-stationary with stationary increments if $1/2\le\al<1.$
\n Observe also that the spectral density (\ref{rb1}) behaves
as $O(|\la|^{-2\al})$ as $|\la|\to 0$ and as $O(|\la|^{-2(\al+\be)})$
as $|\la|\to \f$.
Thus, under the conditions $0<\al<1/2$, $\be>0$ and $\al+\be>1/2,$
the function $f(\la)$ in (\ref{rb1}) is well-defined for both
$|\la|\to 0$ and  $|\la|\to \f$
due to the presence of the component $(1+\la^2)^{-\be}$, $\be>0$,
which is the Fourier transform of the Bessel potential.

Comparing (\ref{lr3}) and (\ref{rb1}), we observe that the
spectral density of fBm is the limiting case as $\be\to0$
that of fRBm with Hurst index $H=\al-1/2.$

\subsection{A Link Between Stationary Processes and the Trace Problem.}
\label{link}
As it was mentioned above, Toeplitz matrices and operators arise
naturally in the theory of stationary processes, and serve as tools,
to study many topics of the spectral and statistical
analysis of discrete- and continuous-time stationary processes.

To understand the relevance of the trace approximation problem
to stationary processes, consider a question concerning the asymptotic
distribution (as $T\to\f$) of the following Toeplitz type quadratic
functionals of a Gaussian stationary process
$\{X(u), \ u\in \mathbb{U}\}$ with spectral density
$f(\la)$, $\la\in \Lambda$ and covariance function
$r(t):=\widehat f(t)$, $t\in \mathbb{U}$ (here $\mathbb{U}$
and $\Lambda$ are as in Section \ref{model}):
\beq
\label{MTc-1}
Q_T:=
 \left \{
 \begin{array}{ll}
\int_0^T\int_0^T\widehat g(t-s)X(t)X(s)\,dt\,ds&
\mbox{in the continuous-time case}\\
\\
\sum_{k=1}^T\sum_{j=1}^T\widehat g(k-j)X(k)X(j) & 
\mbox{in the discrete-time case},
\end{array}
\right. \eeq
where
$
\widehat g(t)=\int_\Lambda e^{i\la t}\,g(\la)\,d\la,\,\,t\in \mathbb{U}
$
is the Fourier transform of some real, even, integrable
function $g(\la),$ $\la\in\Lambda$. We will refer $g(\la)$ as a
generating function for the functional $Q_T$.

The limit distributions of the functionals (\ref{MTc-1}) are
completely determined by the spectral density $f(\la)$ and the
generating function $g(\la)$, and depending on their properties
the limit distributions can be either Gaussian (that is, $Q_T$
with an appropriate normalization obeys central limit theorem),
or non-Gaussian.
The following two questions arise naturally:

\vskip1mm
a) Under what conditions on $f(\la)$ and $g(\la)$ will
the limits be Gaussian?

b) Describe the limit distributions, if they are non-Gaussian.

\vskip1mm
\n
These questions will be discussed in detail in Section \ref{CLT}.

Let $A_T(f)$ be the covariance matrix (or operator) of the process
$\{X(u), \ u\in \mathbb{U}\}$, that is, $A_T(f)$ denote
either the $T\times T$ Toeplitz matrix, or the $T$-truncated
Toeplitz operator generated by the spectral density $f$,
and let $A_T(g)$ denote either the $T\times T$ Toeplitz matrix,
or the $T$-truncated Toeplitz operator generated by the function $g$.

Our study of the asymptotic distribution of the quadratic
functionals (\ref{MTc-1}) is based on the following well--known
results (see, e.g., \cite{GS,I}):
\begin{itemize}
\item[1.]
The quadratic functional $Q_T$ in (\ref{MTc-1}) has the same
distribution as the sum $\sum_{k = 1}^\f \la_k^2\xi_k^2$
($\sum_{k = 1}^T \la_k^2\xi_k^2$ in the discrete-time case),
where $\{\xi_k, k\ge1\}$ are independent $N(0,1)$ Gaussian random variables
and $\{\la_k, k\ge1\}$ are the eigenvalues of the operator $A_T(f)A_T(g)$.
(Observe that the sets of non--zero eigenvalues of the operators
$A_T(f)A_T(g)$, $A_T(g)A_T(f)$ and $A_T^{1/2}(f)A_T(g)A_T^{1/2}(f)$
coincide, where $A_T^{1/2}(f)$ denotes the positive
definite square root of $A_T(f)$.

\item[2.]
The characteristic function $\I (t)$ of $Q_T$ is given by
\begin{equation}
\I (t) = \prod_{k = 1}^\f|1 - 2it\la_k|^{-1/2},
\end{equation}

\item[3.]
The $k$--th order cumulant $\chi_k(\cdot)$  of $Q_T$ is given by
\begin{eqnarray}
\label{MTc-5}
\chi_k(Q_T) = 2^{k-1}(k-1)! \sum_{j = 1}^\f \la_j^k
=2^{k-1}(k-1)!\, \tr\,[A_T(f)A_T(g)]^k
\end{eqnarray}
\end{itemize}

Thus, to describe the asymptotic distributions of the quadratic
functionals (\ref{MTc-1}), we have to control the corresponding
traces of the products of Toeplitz matrices(or operators), yielding
the trace approximation problem with generating functions specified
by (\ref{rim-1}).

\section{The Trace Problem for Toeplitz Matrices}
\label{ATM}

\n
Let $f(\la)$ be an integrable real symmetric function defined on
$\mathbb{T}=(-\pi, \pi]$. For  $T=1, 2,\ldots$ denote by $B_T(f)$
the $(T\times T)$ Toeplitz matrix generated by function $f$, that is,
\beq
\label{MT2-1}
B_T(f):=\|\widehat f(s-t)\|_{s,t=\overline{1,T}}
=\left (\begin{array}{llll}
\widehat f(0)&\widehat f(-1)&\ldots&\widehat f(1-T)\\
\widehat f(1)&\widehat f(0)&\ldots&\widehat f(2-T)\\
\ldots&\ldots&\ldots&\ldots\\
\widehat f(T-1)&\widehat f(T-2)&\ldots&\widehat f(0)
\end{array}\right),
\eeq
where
\beq
\label{FC}
\widehat f(t)=\int_\mathbb{T} e^{i\la t}\,f(\la)\,d\la,
\q t\in \mathbb{Z},
\eeq
are the Fourier coefficients of $f$.

Observe that %for a single matrix $B_T(f)$ we have
\beq
\label{MT-00}
\frac1T\tr\left[B_T(f)\right]=\frac1T\cd T\widehat f(0)
= \inl f(\la)d\la.
\eeq
What happens when the matrix $B_T(f)$ is replaced by
a product of Toeplitz matrices? Observe that the product of
Toeplitz matrices is not a Toeplitz matrix.

The idea is to approximate the trace of the product of
Toeplitz matrices by the trace of a Toeplitz matrix generated
by the product of the generating functions. More precisely,
let $\mathcal{H}=\{h_1,h_2,\ldots,h_m\}$ be a collection of integrable
real symmetric functions defined on $\mathbb{T}$. Define
\bea
\label{im-07}
S_{B,\mathcal{H}}(T):=\frac1T\tr\left[\prod_{i=1}^m B_T(h_i)\right],\q M_{\mathbb{T},\mathcal{H}}:=(2\pi)^{m-1}\int_{-\pi}^{\pi}
\left[\prod_{i=1}^m h_i(\la)\right]\,d\la,
\eea
and let
\bea
\label{im-7}
\De(T):=\De_{B,\mathbb{T},\cal H}(T)=|S_{B,\mathcal{H}}(T)-
M_{\mathbb{T},\mathcal{H}}|.
\eea
Observe that by (\ref{MT-00})
\bea
\label{im-77}
M_{\mathbb{T},\mathcal{H}}=(2\pi)^{m-1}\int_{-\pi}^{\pi}
\left[\prod_{i=1}^m h_i(\la)\right]\,d\la
=\frac1T\tr \left[B_T\left(\prod_{i=1}^m h_i(\la)\right)\right].
\eea
How well is $S_{B,\mathcal{H}}(T)$ approximated by
$M_{\mathbb{T},\mathcal{H}}$? What is the rate of
convergence to zero of approximation error
$\De_{B,\mathbb{T},\cal H}(T)$ as $T\to\f$?
These are Problems (A) and (B).

\subsection{Problem (A) for Toeplitz Matrices}

Recall that Problem (A) involves finding conditions on the
functions $h_1(\la)$, $h_2(\la)$, $\ldots$, $h_m(\la)$ in
(\ref{im-07}) such that
$\De_{B,\mathbb{T},\cal H}(T)=o(1)$ as $T\to\f$.

In Theorem \ref{T1} and Remark \ref{rem2-2} we summarize
the results concerning Problems (A) for Toeplitz matrices
in the case where the exponents $\tau_k, \, k=\overline{1, m}$
(see (\ref{im-1})) are all equal to 1 as in (\ref{im-07}).
\n
\begin{thm}
\label{T1}
Let $\De_{B,\mathbb{T},\cal H}(T)$ be as in (\ref{im-7}).
Each of the following conditions is sufficient for
\beq
\label{in-5}
\De_{B,\mathbb{T},\cal H}(T)=o(1) \q {\rm as} \q T\to\f.
\eeq
\begin{itemize}

\item[{\bf (A1)}]
$h_i\in L^{p_i}(\mathbb{T})$, $1\le p_i\le\f$,
$i=\overline{1, m}$, with $1/p_1+\ldots+1/p_m\le 1$.

\item[{\bf (A2)}]
The function $\varphi({\bf u})$ given by
\beq
\label{in-6}
\varphi({\bf u}):\,=\,
\inl h_1(\la)h_2(\la-u_1)h_3(\la-u_2)\cdots h_m(\la-u_{m-1})\,d\la,
\eeq
where ${\bf u}=(u_1,u_2,\ldots,u_{m-1})\in \mathbb{R}^{m-1}$,
belongs to $L^{m-2}(\mathbb{T}^{m-1})$ and is continuous at
${\bf 0}=(0,0,\ldots,0)$.
\end{itemize}
\end{thm}

\begin{rem}
\label{rem2-1}
{\rm Assertion (A1) was proved by Avram \cite{A}.
For the special case $p_i=\f$, $i=\overline{1, m}$, that is,
when all $h_i$ are bounded functions, it was first established
by Grenander and Szeg\"o (\cite{GS}, Sec. 7.4).
For $m=4$; $p_1=p_3=2$; $p_2=p_4=\f$,
(A1) was proved by Ibragimov \cite{I} and Rosenblatt \cite{R2}.

Assertion (A2), for $m=4$, $h_1=h_3:=f$ and $h_2=h_4:=g$
was proved in Ginovyan and Sahakyan \cite{GS1}.}
\end{rem}

\begin{rem}
\label{rem2-2}
{\rm For the special case $m=4$, $h_1=h_3:=f$ and $h_2=h_4:=g$,
in Giraitis and Surgailis \cite{GSu}
(see also Giraitis et al. \cite{GKSu}),
and in Ginovyan and Sahakyan \cite{GS1}, it was proved that the
following conditions are also sufficient for (\ref{in-5}):

\n {\bf (A3)} (Giraitis and Surgailis \cite{GSu}).
$f\in  L^2(\mathbb{T})$, \,$g\in  L^2(\mathbb{T})$,
$fg\in L^2(\mathbb{T})$ and
$$\inl
 f^2(\la)g^2(\la-\mu)\,d\la \longrightarrow
 \inl f^2(\la)g^2(\la)\,d\la \quad {\rm as} \quad \mu\to0.
$$
\n{\bf (A4)} (Ginovyan and Sahakyan \cite{GS1}).
The functions $f$ and $g$ satisfy
$$
f(\lambda)\le |\lambda|^{-\alpha}L_1(\lambda) \quad \text{and}\quad
|g(\lambda)|\le |\lambda|^{-\beta}L_2(\lambda)\quad \text{for}\quad
\lambda\in [-\pi, \pi],
$$
for some $\alpha<1,\ \beta<1$ with  $\alpha+\beta\le1/2,$
and $L_i\in SV(\mathbb{R})$,
$\ \lambda^{-(\alpha+\beta)}L_i(\lambda)\in L^2(\mathbb{T}),\ \ i=1,2$,
where $SV(\mathbb{R})$ is the class of slowly varying
at zero functions $u(\lambda)$, $\lambda\in\mathbb{R}$,
namely
$\lim_{a\to0}\frac{u(a\lambda)}{u(\lambda)}=1$ for all $a>0$,
satisfying also
$u(\lambda)\in L^\infty(\mathbb{R}),$\
$\lim_{\lambda\to0}u(\lambda)=0,$ \
$u(\lambda)=u(-\lambda)$ and $0<u(\lambda)<u(\mu)$\ for\ $0<\lambda<\mu.$}
\end{rem}

\begin{rem}
\label{MTrem1}
{\rm Case (A4), with $\alpha+\beta<1/2$, was first obtained by Fox and Taqqu \cite{FT1}}.
\end{rem}

\begin{rem}
\label{rem2-12}
{\rm It would be of interest to extend the results of (A3) and (A4) 
to arbitrary $m>4$.}
\end{rem}

We now consider the case when the product in (\ref{im-1})
involves also {\em inverse matrices}, that is,
$\tau_k=(-1)^k, \, k=\overline{1, m}$.
We assume that $m=2\nu$, and the functions from the collection
$\mathcal{H}=\{h_1,h_2,\ldots,h_m\}$ that involve Toeplitz matrices
we denote by $g_i$, $i=\overline{1, \nu}$, while those involving
inverse Toeplitz matrices we denote by $f_i$, $i=\overline{1, \nu}$.
We set
\bea
\label{iim-07}
&&SI_{B,\mathcal{H}}(T):=
\frac1T\tr\left[\prod_{i=1}^\nu[B_T(f_i)]^{-1}B_T(g_i)\right],\\
\label{iim-77}
&&MI_{\mathbb{T},\mathcal{H}}:= 
\frac1{2\pi}\int_{-\pi}^{\pi} \left[\prod_{i=1}^\nu
\frac{g_i(\la)}{f_i(\la)}\right]\,d\la,\\
\label{iim-7}
&&\De I_{B,\mathbb{T},\mathcal{H}}(T):=
|SI_{B,\mathcal{H}}(T)-MI_{\mathbb{T},\mathcal{H}}|.
\eea

The following theorem was proved by Dahlhaus (see \cite{Dah1}, Theorem 5.1).
\begin{thm}
\label{T2}
Let $\nu\in\mathbb{N}$, and $\alpha, \beta \in\mathbb{R}$
with $0<\alpha, \beta<1$ and $\nu(\beta-\alpha)<1/2$.
Suppose $f_i(\la)$ and $g_i(\la)$, $i=\overline{1, \nu}$, are symmetric
real valued functions satisfying the conditions:
\begin{itemize}
\item[{\rm (C1)}]
$f_i(\la)$, $i=\overline{1, \nu}$, are nonnegative and continuous 
at all $\la\in \mathbb{T}\setminus\{0\}$, \
$f_i^{-1}(\la)$ are continuous at all $\la\in \mathbb{T}$, and
$f_i(\lambda)=O\left( |\lambda|^{-\alpha-\de}\right)$
as $\la\to0$ for all $\de>0$ and  $i=\overline{1, \nu}$;
\item[{\rm (C2)}]
$\partial/(\partial\la) f_i^{-1}(\la)$ and
$\partial^2/(\partial\la)^2 f_i^{-1}(\la)$
are continuous at all $\la\in \mathbb{T}\setminus\{0\}$,
and for all $\de>0$ and  $i=\overline{1, \nu}$
$$
\partial^k /(\partial\la)^k f_i^{-1}(\la)=
O\left( |\lambda|^{-\alpha-k-\de}\right) \q {\rm for} \q k=0,1.
$$
\item[{\rm (C3)}]
$g_i(\la)$ are continuous at all $\la\in \mathbb{T}\setminus\{0\}$ and
$g_i(\lambda)=O\left( |\lambda|^{-\beta-\de}\right)$ as $\la\to0$ for all $\de>0$
and $i=\overline{1, \nu}$.
\end{itemize}
Then
\beqq
\De I_{B,\mathbb{T},\mathcal{H}}(T)=o(1) \q {\rm as} \q T\to\f.
\eeqq
\end{thm}

\subsection{Problem (B) for Toeplitz Matrices}
\label{BTM}
Recall that Problem (B) involves finding conditions on the
functions $h_1(\la)$, $h_2(\la)$, $\ldots$, $h_m(\la)$ in
(\ref{im-07}) such that
$\De_{B,\mathbb{T},\cal H}(T)=O(T^{-\g})$
as $T\to\f$ for some $\g>0$.

In Theorem \ref{T3} below we summarize the results concerning
Problem (B) for Toeplitz matrices in the case where
$\tau_k=1, \, k=\overline{1, m}$.
First we introduce some classes of functions
(see, e.g.,  \cite{BN,Ta,TK}). Recall that $\mathbb{T}=(-\pi,\pi]$
and denote
\beq
\label{m-1}
\mathcal{F}_1(\mathbb{T}):= \left\{f\in L^1(\mathbb{T}): \, \sum_{k=-\f}^\f|k||
\widehat f(k)|<\f; \right\},
\eeq
where $\widehat f(k)=\int_\mathbb{T} e^{i\la k}\,f(\la)\,d\la,$ \
$\widehat f(-k)=\widehat f(k)$, and
\beq
\label{m-2}
\mathcal{F}_2(\mathbb{T})=\mathcal{F}_{ARMA}(\mathbb{T}):=
\left\{f: \, f(\la)=\frac{\si^2}{2\pi}\left|\frac{A(e^{i\la})}{B(e^{i\la})}
\right|^2\right\},
\eeq
where $0<\sigma^2<\f$, $A(z):=\sum_{k=0}^qa_kz^k$ ($p\in \mathbb{N}$) and
$B(z):=\sum_{k=0}^pb_kz^k$ ($q\in \mathbb{N}$)
are both bounded away from zero for $|z|\le1.$
\begin{rem}
\label{rem3-1}
{\rm The following implications were established in \cite{Ta}:
\begin{itemize}
\item[(a)]
If $f_1, f_2 \in \mathcal{F}_1(\mathbb{T})$, then
$f_1f_2 \in \mathcal{F}_1(\mathbb{T})$.
\item[(b)]
If $f\in \mathcal{F}_2(\mathbb{T})$, then $f \in \mathcal{F}_{1}(\mathbb{T})$
and $f^{-1} \in \mathcal{F}_2(\mathbb{T})$.
\end{itemize}}
\end{rem}

For $\psi\in L^p(\mathbb{T})$, $1\le p\le\infty$, let
$\omega_p(\psi,\delta)$ denote the $L^p$--modulus of continuity of $\psi$:
$$
\omega_p(\psi,\delta) :=\sup_{0<h\le \delta}\|\psi(\cd+h)-\psi(\cd)\|_p,
\quad \delta>0.
$$
\begin{den}
{\rm
Given numbers $0 < \g \le 1$ and $1\le p\le\infty$, we denote by
$\Lip (\mathbb{T}; p, \g)$ the $L^p$-Lipschitz
class of functions defined on $\mathbb{T}$ (see, e.g., \cite{BN}):
$$\Lip (\mathbb{T};p, \g) = \{\psi(\la)\in L^p(\mathbb{T});
\q \om_p(\psi;\de) = O(\de ^\g),
\q \de \to 0\}.$$
}
\end{den}
Observe that if $\psi\in\Lip (p, \g)$, then there exists a
constant $C$ such that $\om_p(\psi;\de)\le C\,\de ^\g$ for all
$\de>0$.

\begin{thm}
\label{T3}
Assume that $\tau_k=1, \, k=\overline{1, m}$,
$\mathcal{H}=\{h_1,h_2,\ldots,h_m\}$,
and let $\De_{B,\mathbb{T},\cal H}(T)$ be as in (\ref{im-7}).
The following assertions hold:
\begin{itemize}
\item[{\bf (B1)}]
If $h_i \in \mathcal{F}_1(\mathbb{T})$, $i=\overline{1, m}$, then
\beqq
\De_{B,\mathbb{T},\cal H}(T)=O(T^{-1}) \q  {\rm as }\ T\to\f.
\eeqq

\item[{\bf (B2)}]
If the functions $h_i(\la)$, $i=\overline{1, m}$,  have uniformly
bounded derivatives
on $\mathbb{T}:=(-\pi,\pi]$, then for any \ $\epsilon>0$
\begin{equation*}
\De_{B,\mathbb{T},\cal H}(T)=O(T^{-1+\epsilon}) \q  as \,\, T\to\f.
\end{equation*}

\item[{\bf (B3)}]
Assume that the function $\varphi({\bf u})$ given by (\ref{in-6})
with some constants $C>0$ and $\g\in(0,1]$ satisfies
$$
 |\varphi ({\bf
u})-\varphi ({\bf 0})|\leq C|{\bf u}|^\g, \q
{\bf u}=(u_1,u_2,\ldots,u_{m-1})\in \mathbb{T}^{m-1},
$$
 where
${\bf 0}=(0,0,\ldots,0)$ and $|{\bf u}|=|u_1|+ |u_2| +\cdots+|u_{m-1}|$.
Then for any $\varepsilon >0$
\beq
\label{t2}
\De_{B,\mathbb{T},\cal H}(T)=O\left(T^{-\g+\varepsilon}\right)
\quad\text{as}\quad T\to\infty. \eeq

\item[{\bf (B4)}]
Let $h_i(\la)\in\Lip(\mathbb{T}; p_i, \g)$, $p_i>1$, $i=1,2,\ldots,m$,
$1/p_1+\cdots+1/p_m\leq 1$ and $\g \in (0,1]$.
Then (\ref{t2}) holds for any $\varepsilon >0$.

\item[{\bf (B5)}]
Let $h_i(\lambda)$, $i=1,2,\ldots,m$, be differentiable functions
defined on $\mathbb{T}\setminus\{0\}$, such that for some constants
$C_{1i}>0$, $C_{2i}>0$, and
$\alpha_i$, $i=1,2,\ldots,m$, satisfying $0<\alpha_i<1,$ \
$\alpha:=\sum_{i=1}^m\alpha_i<1$
\begin{equation*}
|h_i(\lambda)|\le C_{1i}|\lambda|^{-\alpha_i},\quad
|h_i^\prime(\lambda)|\le C_{2i}|\lambda|^{-(\alpha_i+1)},\quad
\lambda\in \mathbb{T}\setminus\{0\}, \,\, i=\ol{1,m}.
\end{equation*}
Then (\ref{t2}) holds for any $\varepsilon >0$ with
\begin{equation}
\label{d2}
\gamma =\frac1m(1-\alpha).
\end{equation}
\end{itemize}
\end{thm}

\begin{rem}
\label{rem3-2}
{\rm Assertion (B1) was proved in Taniguchi \cite{Ta}
(see, also, \cite{TK}). Assertion (B2), which is weaker than (B1),
but under weaker conditions than those in (B1), was proved in
Lieberman and Phillips \cite{LP}.
Assertions (B3)--(B5) for $m=4$ were proved in Ginovyan and
Sahakyan \cite{GS6}.}
\end{rem}

\begin{rem}
\label{rem3-3}
{\rm It is easy to see that under the conditions of {\bf(B2)}
we have $h_i\in{\rm Lip}(\mathbb{T}; p, 1)$ for any $i=1,2,\ldots m$
and $p\ge 1$. Hence {\bf(B4)} implies {\bf(B2)}.}
\end{rem}

\begin{exa}
\label{exa1}
{\rm Let $h_i(\lambda)=|\lambda|^{-\alpha_i}$, $\lambda\in [-\pi,\pi],$
$i=1,2,\ldots m$, with $0<\alpha_i<1$ and $\alpha:=\sum_{i=1}^m\alpha_i<1$.
It is easy to see that the conditions of {\bf(B5)} are satisfied,
and hence we have (\ref{t2}) with $\gamma $ as in (\ref{d2}).}
\end{exa}

The next results (cf. Ginovyan \cite{G4}) show that for special case $m=2$ the rates
in Theorem \ref{T3} (B4) and (B5) can be substantially improved.
\begin{thm}
\label{T4-1}
Let $h_i(\la)\in\Lip(\mathbb{T}; p_i, \g_i)$ with $p_i>1$,
$1/p_1+1/p_2 = 1$ and $\g_i \in (0,1],$ $i=1,2$, and let
\begin{equation*}
\De_{2, B}(T):=
\left|\frac1T\tr[B_T(h_1)B_T(h_2)]- 2\pi
\int_{\mathbb{T}}h_1(\la)h_2(\la)\,d\la\right|.
\end{equation*}
Then
\begin{equation*}
\De_{2, B}(T)= \left \{
           \begin{array}{lll}
           O(T^{-(\g_1+\g_2)}), & \mbox {if \, $\g_1+\g_2<1$}\\
           O(T^{-1}\ln T), & \mbox {if \, $\g_1+\g_2=1$}\\
           O(T^{-1}), & \mbox {if \, $\g_1+\g_2>1$}.
           \end{array}
           \right.
\end{equation*}
\end{thm}

\begin{thm}
\label{T4-2}
Assume that the functions $h_i(\lambda)$, $i=1,2$, satisfy the conditions
of Theorem \ref{T3} (B5) with $m=2$.
Then
\begin{equation}
\label{t020}
\De_{2, B}(T)= O\left(T^{-1+(\al_1+\al_2)}\right)
\quad\text{as}\quad T\to\infty.
\end{equation}
\end{thm}

The next result, due to Taniguchi \cite{Ta} (see, also, \cite{TK}),
concerns the case when the product in (\ref{im-1}) involves
also inverse matrices, that is, $\tau_k=(-1)^k, \, k=\overline{1, m}$.

\begin{thm}
\label{T4}
Let $SI_{B,\mathcal{H}}(T)$, $MI_{\mathbb{T},\mathcal{H}}$ and
$\De I_{B,\mathbb{T},\mathcal{H}}(T)$ be as in (\ref{iim-07}),
and let $\mathcal{F}_1(\mathbb{T})$ and $\mathcal{F}_2(\mathbb{T})$
be as in (\ref{m-1}) and (\ref{m-2}), respectively.
If $f_i \in \mathcal{F}_{2}$ and $g_i \in \mathcal{F}_1(\mathbb{T})$,
$i=\overline{1, \nu}$, then
\beq
\label{iim-9}
\De I_{B,\mathbb{T},\mathcal{H}}(T)=O(T^{-1}) \q  {\rm as }\ T\to\f.
\eeq
\end{thm}

\section{The Trace Problem for Toeplitz Operators}
\label{ATO}
In this section we consider Problems (A) and (B) for Toeplitz operators,
that is, in the case where the generating functions are defined on the
real line. Again, Problem (A) involves $o(1)$ approximation and
Problem (B) involves $O(T^{-\g})$ approximation with $\g>0$.
The theorems in this section are  proved in Section \ref{PR}.

Let $f(\la)$ be an integrable real symmetric function defined on
$\mathbb{R}$.
The analogue of the Fourier coefficients $\widehat f(k)$
in (\ref{FC}) is the Fourier transform $\widehat f(t)$ of $f(\la)$:
\beq
\label{FT}
\widehat f(t)=\itf e^{i\la t}\,f(\la)\,d\la, \q
t\in \mathbb{R}.
\eeq
The $\widehat f$ in (\ref{FT}) will play the role of
kernel in an integral operator.

Given $T>0$ and an integrable real symmetric function $f(\la)$ defined on
$\mathbb{R}$, the {\it $T$-truncated Toeplitz operator\/} generated by $f(\la)$,
denoted by $W_T(f)$, is defined by the following equation
(see, e.g., \cite{GS2,GS,I}):
\beq
\label{MT3-1}
[W_T(f)u](t)=\int_0^T\hat f(t-s)u(s)ds,
\q u(s)\in L^2[0,T],
\eeq
where $\widehat f$ is as in (\ref{FT}).

It follows from (\ref{FT}), (\ref{MT3-1}) and the formula for traces of
integral operators (see, e.g., \cite{GK}, p. 114) that
\beq
\label{MT3-2}
\tr\left[W_T(f)\right]= \int_0^T\hat f(t-t)dt= T\hat f(0) = T\itf f(\la)d\la.
\eeq
We pose the same question as in the case of Toeplitz matrices:
what happens when the single operator $W_T(f)$ is replaced by
a product of such operators? Observe that the product of
Toeplitz operators again is not a Toeplitz operator.

The approach is similar to that of Toeplitz matrices -
to approximate the trace of the product of Toeplitz operators
by the trace of a Toeplitz operator generated
by the product of generating functions. More precisely,
let $\mathcal{H}=\{h_1,h_2,\ldots,h_m\}$ be a collection of integrable
real symmetric functions defined on $\mathbb{R}$. Define
\bea
\label{n 4-4}
S_{W,\mathcal{H}}(T):=\frac1T\tr\left[\prod_{i=1}^m W_T(h_i)\right],\q M_{\mathbb{R},\mathcal{H}}:=(2\pi)^{m-1}\int_{-\f}^{\f}
\left[\prod_{i=1}^m h_i(\la)\right]\,d\la,
\eea
and let
\bea
\label{n 4-5}
&&\De(T):=\De_{W,{\mathbb R},\cal H}(T)=|S_{W,\mathcal{H}}(T)-
M_{{\mathbb R},\mathcal{H}}|.
\eea
Observe that by (\ref{MT3-2}),
\bea
\label{im-77}
M_{\mathbb{R},\mathcal{H}}=(2\pi)^{m-1}\int_{-\f}^{\f}
\left[\prod_{i=1}^m h_i(\la)\right]\,d\la
=\frac1T\tr \left[W_T\left(\prod_{i=1}^m h_i(\la)\right)\right].
\eea
How well is $S_{W,\mathcal{H}}(T)$ approximated by 
$M_{\mathbb{R},\mathcal{H}}$? What is the rate of
convergence to zero of approximation error
$\De_{W,\mathbb{R},\cal H}(T)$ as $T\to\f$?
These are Problems (A) and (B) in this case.

\subsection{Problem (A) for Toeplitz Operators}

\n
In Theorem \ref{T5} and Remark \ref{rem4-1} we summarize the results
concerning Problems (A) for Toeplitz operators in the case where
$\tau_k=1, \, k=\overline{1, m}$.
\begin{thm}
\label{T5}
Let $\De(T):=\De_{W,{\mathbb R},\cal H}(T)$ be as in (\ref{n 4-5}).
Each of the following conditions is sufficient for
\beq
\label{n 4-7}
\De(T)=o(1) \q {\rm as} \q T\to\f.
\eeq
\begin{itemize}
\item[{\bf (A1)}]
$h_i\in L^1(\mathbb{R})\cap L^{p_i}(\mathbb{R})$, $p_i>1$, $i=\overline{1, m}$,
with $1/p_1+\ldots+1/p_m\le 1$.

\item[{\bf (A2)}]
The function $\varphi({\bf u})$ defined by
\beq
\label{n 4-6}
\varphi({\bf u}):\,=\,
\itf h_1(\la)h_2(\la-u_1)h_3(\la-u_2)\cdots h_m(\la-u_{m-1})\,d\la,
\eeq
where ${\bf u}=(u_1,u_2,\ldots,u_{m-1})\in \mathbb{R}^{m-1}$,
belongs to $L^{m-2}(\mathbb{R}^{m-1})$ and is continuous at
${\bf 0}=(0,0,\ldots,0)\in \mathbb{R}^{m-1}$.
\end{itemize}
\end{thm}

\begin{rem}
\label{rem4-1}
{\rm For the special case $m=4$, $h_1=h_3:=f$ and $h_2=h_4:=g$,
Ginovyan and Sahakyan \cite{GS2} prove that the following
conditions are also sufficient for (\ref{n 4-7}):

\n {\bf (A3)}
$f\in L^1(\mathbb{R})\cap L^2(\mathbb{R})$,
\,$g\in L^1(\mathbb{R})\cap L^2(\mathbb{R})$,
\,$fg\in L^2(\mathbb{R})$ and
\beqq
\itf f^2(\la)g^2(\la-\mu)\,d\la \longrightarrow
 \itf f^2(\la)g^2(\la)\,d\la \quad {\rm as} \quad \mu\to0.
\eeqq
{\bf (A4)}
The functions $f$ and $g$ are integrable on $\mathbb{R}$,
bounded on $\mathbb{R}\setminus (-\pi, \pi)$, and satisfy
$$
f(\lambda)\le |\lambda|^{-\alpha}L_1(\lambda) \quad \text{and}\quad
|g(\lambda)|\le |\lambda|^{-\beta}L_2(\lambda)\quad \text{for} \quad
\lambda\in [-\pi, \pi],
$$
for some $\alpha<1,\ \beta<1$ with  $\alpha+\beta\le1/2,$
and $L_i\in SV(\mathbb{R})$,
$\ \lambda^{-(\alpha+\beta)}L_i(\lambda)\in L^2(\mathbb{T}),\ \ i=1,2$,
where $SV(\mathbb{R})$ is the class of slowly varying
at zero functions $u(\lambda)$, $\lambda\in\mathbb{R}$, satisfying
$u(\lambda)\in L^\infty(\mathbb{R}),$\
$\lim_{\lambda\to0}u(\lambda)=0,$ \
$u(\lambda)=u(-\lambda)$ and $0<u(\lambda)<u(\mu)$\ for\ $0<\lambda<\mu.$
}
\end{rem}

\begin{rem}
\label{rem2-122}
{\rm It would be of interest to extend the results of (A3) and
(A4) to arbitrary $m>4$.}
\end{rem}

\subsection{Problem (B) for Toeplitz Operators}

\n
In Theorem \ref{T5-1} below we summarize the results concerning
Problems (B) for Toeplitz operators in the case where
$\tau_k=1, \, k=\overline{1, m}$. Let
\bea
\label{mr-1}
\mathcal{F}_1(\mathbb{R}):= \left\{f\in L^1(\mathbb{R}): \, \int_{-\f}^\f|t|
|\widehat f(t)|dt<\f \right\},
\eea
where $\widehat f(t)=\int_\mathbb{R} e^{i\la t}\,f(\la)\,d\la$, \
$\widehat f(-t)=\widehat f(t)$.

For $\psi\in L^p(\rr)$, $1\le p\le\infty$
let
$\omega_p(\psi,\delta)$ denote the $L^p$--modulus of continuity of $\psi$:
$$
\omega_p(\psi,\delta) :=\sup_{0<h\le \delta}\|\psi(\cd+h)-\psi(\cd)\|_p,\q \delta>0.
$$
Given numbers $0 < \g \le 1$ and $1\le p\le\infty$, we denote by
$\Lip (\mathbb{R}; p, \g)$ the $L^p$-Lipschitz
class of functions defined on $\rr$ (see, e.g.,  \cite{BN}):
$$
\Lip (\mathbb{R}; p, \g) = \{\psi(\la)\in L^p(\rr)\ : \
\om_p(\psi;\de) = O(\de ^\g)\ \text{as}\
  \de \to 0\}.$$

Theorem \ref{T5-1} is the continuous version of  Theorem \ref{T3}.
\begin{thm}
\label{T5-1}
Let ${\cal H}=\{h_1,h_2,\ldots,h_m\}$, and $\De_{{W,\mathbb R},\cal H}(T)$
and $\varphi({\bf u})$ be as in (\ref{n 4-5}) and (\ref{n 4-6}), respectively.
The following assertions hold:
\begin{itemize}
\item[{\bf (B1)}]
If $h_i \in \mathcal{F}_1(\mathbb{R})$, $i=\overline{1, m}$, then
\beq
\label{mr-2 new}
\De_{W,{\mathbb R},\cal H}(T)=O(T^{-1}) \q  {\rm as }\q T\to\f.
\eeq

\item[{\bf (B2)}]
Assume that $\varphi ({\bf u})\in L^\infty(\rr^{m-1})$
and with some constants $C>0$ and $\g\in(0,1]$
\beq
\label{n 5-3}  |\varphi ({\bf
u})-\varphi ({\bf 0})|\leq C|{\bf u}|^\g, \q
{\bf u}=(u_1,u_2,\ldots,u_{m-1})\in \rr^{m-1},
\eeq where
${\bf 0}=(0,0,\ldots,0)$ and $|{\bf u}|=|u_1|+ |u_2| +\cdots+|u_{m-1}|$.
Then for any $\varepsilon >0$
\beq
\label{tt2 new}
\De_{W,{\mathbb R},\cal H}(T)=O\left(T^{-\g+\varepsilon}\right)
\quad\text{as}\quad T\to\infty. \eeq

\item[{\bf (B3)}]
Let $h_i(\la)\in\Lip(\mathbb{R}; p_i, \g)$, $i=1,2,\ldots,m$,
$1/p_1+\cdots+1/p_m\leq 1$ and $\g \in (0,1]$.
Then (\ref{tt2 new}) holds for any $\varepsilon >0$.

\item[{\bf (B4)}]
Let $h_i(\lambda)$, $i=1,2,\ldots,m$, be differentiable functions
defined on $\mathbb{R}\setminus\{0\}$, such that for some
constants $C_{i}>0$ and
$\sigma_i>0$, $\delta_i>1$, $i=1,2,\ldots,m$ with
$\sigma:=\sum_{i=1}^m\sigma_i<1$
\begin{equation}
\label{prime new}
|h_i(\lambda)|\le
\begin{cases}
C_{i}|\lambda|^{-\sigma_i}& \text {if}\q |\lambda|\le 1\cr
C_{i}|\lambda|^{-\delta_i}& \text {if} \q |\lambda|>1
\end{cases}, \quad
|h_i^\prime(\lambda)|\le
\begin{cases}
C_{i}|\lambda|^{-\sigma_i-1}& \text {if} \q |\lambda|\le 1\cr
C_{i}|\lambda|^{-\delta_i-1}& \text {if} \q |\lambda|>1
\end{cases}
\end{equation}
for all \, $i=1,2,\ldots,m$. Then for any $\varepsilon >0$
\beq
\label{n 5-6}
\Delta_{W,{\mathbb R},{\cal H}}(T)=O\left(T^{-\g+\varepsilon}\right)
\quad\text{as}\quad T\to\infty
\eeq
with
\begin{equation}
\label{n 5-7}
\gamma =\frac1m(1-\sigma).
\end{equation}
\end{itemize}
\end{thm}

The next results (see \cite{G4}), which are continuous versions of Theorems
\ref {T4-1} and \ref {T4-2}, respectively, show that for the special case
$m=2$, the rates in Theorem \ref{T5-1} (B3) and (B4) can be substantially improved.

\begin{thm}
\label{T5-2}
Let $h_i(\la)\in\Lip(\mathbb{R}; p_i, \g_i)$ with $p_i>1$,
$1/p_1+1/p_2 = 1$ and $\g_i \in (0,1],$ $i=1,2$, and let
$$
\De_{2, W}(T):=
\left|\frac1T\tr[W_T(h_1)W_T(h_2)]
- 2\pi\int_{\mathbb{R}}h_1(\la)h_2(\la)\,d\la\right|.
$$
Then
$$
\De_{2, W}(T)= \left \{
           \begin{array}{lll}
           O(T^{-(\g_1+\g_2)}), & \mbox {if \, $\g_1+\g_2<1$}\\
           O(T^{-1}\ln T), & \mbox {if \, $\g_1+\g_2=1$}\\
           O(T^{-1}), & \mbox {if \, $\g_1+\g_2>1$}.
           \end{array}
           \right.
$$
\end{thm}
\begin{thm}
\label{T5-3}
Assume that the functions $h_i(\lambda)$, $i=1,2$, satisfy the conditions
of Theorem \ref{T5-1} (B4) with $m=2$.
Then
\begin{equation}
\label{t53}
\De_{2, W}(T)= O\left(T^{-1+(\sigma_1+\sigma_2)}\right)
\quad\text{as}\quad T\to\infty.
\end{equation}
\end{thm}

\begin{rem}
{\rm It would be of interest to prove the continuous analogs of
Dahl\-haus theorem (Theorem \ref{T2}) and Taniguchi theorem
(Theorem \ref{T4}) for Toeplitz operators.}
\end{rem}

\section{Applications to Stationary Processes}
\label{ASP}
In this section we provide some applications of the trace problem
to discrete- and continuous-time stationary processes:
ARFIMA and Fractional Riesz-Bessel motions;
central and non-central limit theorems, Berry-Ess\'een bounds,
and large deviations for Toeplitz quadratic forms and functionals.

\subsection{Applications to ARFIMA time series and fractional Riesz-Bessel motions}

\n In this subsection we apply the results of Sections \ref{ATM} and \ref{ATO}
to the important special cases where the generating functions
are spectral densities of a discrete-time ARFIMA$(0,d,0)$ stationary
processes or continuous-time stationary fractional Riesz-Bessel motions.

We use the following notation: $m=2\nu$;
\beaa
&&h_1(\la)= h_3(\la)=\cdots =h_{2\nu-1}(\la):=f_1(\la)\\
&&h_2(\la)= h_4(\la)=\cdots =h_{2\nu}(\la):=f_2(\la);
\eeaa
and
\bea
\label{m4-4}
&&S_{\nu,A}(T)=\frac1T\tr[A_T(f_1)A_T(f_2)]^\nu,\\
\label{m4-5}
&&\Delta_{\nu,A}(T):=
\left|S_{\nu,A}(T)- (2\pi)^{2\nu-1}\int_{\Lambda}[f_1(\la)f_2(\la)]^\nu\,d\la\right|,
\eea
where either $A_T(f_i)=B_T(f_i)$ and $\Lambda=\mathbb{T}$
or $A_T(f_i)=W_T(f_i)$ and $\Lambda=\mathbb{R}$, $i=1,2$.

\subsubsection{Applications to ARFIMA time series}

The next theorem gives an error bound for $\De_{2,B}(T)$
in the case where the corresponding Toeplitz matrices are
generated by spectral densities of
two discrete-time ARFIMA$(0,d,0)$ stationary processes.
\begin{thm}
\label{FM}
Let $f_i(\la)$, $i=1,2$, be the spectral density functions
of two ARFIMA$(0,d,0)$ stationary processes defined as
\beq
\label{fs1}
f_i(\la)=\frac{\sigma^2_i}{2\pi}\left|1-e^{i\lambda}\right|^{-2d_i}, \q i=1,2
\eeq
with $0<\sigma^2_i<\infty$ and $0<d_i<1/2$. %, $i=1,2$.
Then under  $d:\,=d_1+d_2<1/(2\nu)$, $\nu\in\mathbb{N}$, for any $\varepsilon >0$,
\begin{equation}
\label{t20}
\Delta_{\nu,B}(T)=O\left(T^{-\gamma+\varepsilon}\right)
\quad\text{as}\quad T\to\infty
\end{equation}
with
\begin{equation}
\label{d21}
\g=\frac1{2\nu}-(d_1+d_2).
\end{equation}
\end{thm}

\begin{proof} Assuming  that $\lambda\in (0,\pi]$
(the case $\lambda\in [-\pi,0)$ is treated similarly),
and taking into account $|1-e^{i\lambda}|=2\sin(\lambda/2)$,
we have for $i=1,2$
\bea
\nonumber
&&f_i(\lambda)= \frac{\sigma^2_i}{2\pi} \cdot 2^{-2d_i }
\left[\sin\frac \lambda 2\right]^{-2d_i},\\
&&f_i^\prime(\lambda)=\frac{\sigma^2_i}{2\pi} \cdot \left[-2d_i2^{-2d_i-1}
\label{t200}
\left(\sin\frac \lambda 2\right)^{-2d_i-1}\cos\frac \lambda 2\right].
\eea
It is clear that the conditions of Theorem \ref{T3} {\bf(B5)} are satisfied
with  $\alpha_i=2d_i$ and $C_{1i}=C_{2i} =\sigma_i^2$,  $i=1,2$,
and the result follows.
\end{proof}

The next theorem, which was proved in Lieberman and Phillips \cite{LP},
gives an explicit second-order asymptotic expansion for $S_{1,B}(T)$
in the case where the Toeplitz matrices are generated by the spectral
densities given by (\ref{fs1}), and shows that in this special case
a second-order asymptotic expansion successfully removes the singularity
and delivers a substantially improved approximation.

\begin{thm}
\label{FM5}
Let $f_i(\la)$, $i=1,2,$ be the spectral density functions
of two ARFIMA$(0,d,0)$ stationary processes defined by (\ref{fs1})
with $0<\sigma^2_i<\infty$ and $0<d_i<1/2$, $i=1,2$.
Then under  $d:\,=d_1+d_2<1/2$
\bea
\label{fs2}
\nonumber
S_{1,B}(T):&=&\frac1T\tr[B_T(f_1)B_T(f_2)]\\
&=&2\pi\inl f_1(\la)f_2(\la)\,d\la -
\frac{C(d_1,d_2)}{T^{1-2d}}
+o\left(\frac{1}{T^{1-2d}}\right)
\eea
as $T\to\f$, where
\beq
\label{fs3}
C(d_1,d_2)=
\frac{2\si_1^2\si_2^2\pi^2}{\cos(\pi d_1)\cos(\pi d_2)\G(2d_1)\G(2d_2)}
\cdot\frac{1}{2d(1-2d)}.
\eeq
\end{thm}

\begin{rem}
\label{FM-rem}
{\rm The asymptotic relation (\ref{fs2}) in Lieberman and Phillips
\cite{LP} was established by direct calculations using the explicit
forms of functions $f_i$ given by (\ref{fs1}). On the other hand,
as it follows from (\ref{t200}), the functions $f_{i}$ ($i=1,2$)
satisfy conditions of Theorem \ref{T4-2} with $\al_i=2d_i$ ($i=1,2$),
and hence (\ref{fs2}) is a special case of Theorem  \ref{T4-2}.}
\end{rem}

\subsubsection{Applications to fractional Riesz-Bessel motions}

Now we assume that the underlying model is a continuous-time
stationary process specified by a fractional Riesz-Bessel motion.
The following result is an immediate consequence of Theorem \ref{T5} {\bf(A1)}.
\begin{thm}
\label{M51}
Let $f_1(\la)=f(\la)$ be the spectral density
of a fractional Riesz-Bessel motion defined by (\ref{rb1}),
and let $f_2(\la)=g(\la)$ be an integrable real symmetric function on
$\mathbb{R}$.
If for some $p, q\ge 1$ with $1/p+1/q\le1/\nu$ ($\nu\in\mathbb{N}$) we have
$g(\la)\in L^q(\mathbb{R})$ and $0<\al<1/(2p),$ $\al+\be>1/{2},$
then
$$
\De_{\nu,W}(T)=o(1) \q {\rm as} \q T\to\f.
$$
\end{thm}

\begin{thm}
\label{M52}
Let $f_1(\la)=f(\la)$ be as in (\ref{rb1}) with $0<\al<1/({2p})$ and
$\al+\be>1/{2}$, and let $f_2(\la)=g(\la)$ be an integrable real
symmetric function from the
class $\Lip(q,1/p-2\al)$ with $1/p+1/q\le1/\nu$, $\nu\in\mathbb{N}$.
Then for any $\varepsilon >0$
\begin{equation}
\label{t020}
\Delta_{\nu,W}(T)=O\left(T^{-\gamma+\varepsilon}\right)
\quad\text{as}\quad T\to\infty.
\end{equation}
with $\g=1/p-2\al$.
\end{thm}

\begin{proof}
The result follows from Theorem \ref{T5-1} {\bf(B3)} and
the following lemma, which is proved in the Appendix.
\begin{lem}
\label{R1}
Let $p>1$,  $0<\sigma<1/{p}$  and let $f(\la)$ be
differentiable function defined on $\mathbb R\setminus\{0\}$, such
that for some constant  $C >0$
\begin{equation}\label {prime1}
|f(\lambda)|\le
\begin{cases}
C|\lambda|^{-\sigma}& \text {if} \q |\lambda|\le 1\cr
C|\lambda|^{-\de}& \text {if} \q |\lambda|>1
\end{cases}, \quad
|f^\prime(\lambda)|\le
\begin{cases}
C|\lambda|^{-\sigma-1}& \text {if} \q |\lambda|\le 1\cr
C|\lambda|^{-\de-1}& \text {if} \q |\lambda|>1
\end{cases}.
\end{equation}
Then $f\in \Lip(p,1/p-\sigma)$.
\end{lem}

Now to prove (\ref{t020}) observe that by (\ref{rb1}),
\beq
\label{rb10}
f^\prime(\la)=-\frac{2\al+2(\al+\be)\la^2}{\la^{2\al+1}(1+\la)^{\be+1}}.
\eeq
It follows from (\ref{rb1}) and (\ref{rb10}) that the functions
 $f$ and $f^\prime$  satisfy conditions (\ref{prime1})
 with $\sigma=2\al$  and $\delta =2\alpha +2\beta$.
Hence by Lemma \ref{R1}, $f(\la)\in\Lip(p, 1/p-2\al).$
Therefore the result follows from Theorem \ref{T5-1} {\bf(B3)}.
\end{proof}

\begin{thm}
\label{M53}
Let $f_i(\la)$, $i=1,2,$ be the spectral density functions
of two fractional Riesz -Bessel motions defined as
\beq
\label{s1}
 f_i(\la)=\frac {C_i} {|\la|^{2\al_i}{(1+\la^2)^{\be_i}}}, \q
 0<\al_i<1/2, \q \al_i+\be_i>1/2, \q i=1, 2.
 \eeq
If \  $\al_1+\al_2<1/({2\nu})$, then (\ref{t020}) holds for any
$\varepsilon >0$ with
\beq \label{tt3}
\g=\frac1{2\nu}-(\al_1+\al_2).
\eeq
\end{thm}

\begin{proof}
It follows from (\ref{rb1}) and (\ref{rb10}) that the functions
$f_1$ and $f_2$  satisfy conditions of Theorem  \ref{T5-1} {\bf(B4)}
with $\sigma_i=2\alpha_i$ and $\delta_i=2\alpha_i+2\beta_i$, $i=1,2$,
and the result follows.
\end{proof}

The next theorem, which is a continuous version of Theorem \ref{FM5},
contains an explicit second-order asymptotic expansion for $S_{1,W}(T)$
in the case where the Toeplitz operators are generated by the spectral
densities given by (\ref{s1}), and shows that in this special case
a second-order asymptotic expansion successfully removes the singularity
and delivers a substantially improved approximation.
\begin{thm}
\label{M5}
Let $f_i(\la)$, $i=1,2,$ be as in (\ref{s1}).
Then under  $\al:\,=\al_1+\al_2<1/2$
 \bea
 \label{s2}
 \nonumber
 S_{1,W}(T):&=&\frac1T\tr[W_T(f_1)W_T(f_2)]\\
 &=&2\pi\itf f_1(\la)f_2(\la)\,d\la -
 \frac{C(\al_1,\al_2)}{T^{1-2\al}}
 +o\left(\frac{1}{T^{1-2\al}}\right)
 \eea
 as $T\to\f$, where
 \beq
 \label{s3}
 C(\al_1,\al_2)=
 \frac{2 C_1C_2\pi^2}{\cos(\pi\al_1)\cos(\pi\al_2)\G(2\al_1)\G(2\al_2)}
 \cdot\frac{1}{2\al(1-2\al)}.
 \eeq
 \end{thm}

The proof is based on the following lemma, which contains an
asymptotic formula for the covariance function of a fRBm process.
It is proved in the Appendix.

\begin{lem}
\label{sl1}
Let $f(\la)$ be as in (\ref{rb1}) with $0<\al<1/2$ and $\be>1/2$,
and let $r(t):\,=\hat f(t)$ be the Fourier transform of $f(\la)$.
Then
\beq
\label{s4}
r(t)=t^{2\al-1}\cd\frac{\pi C}{\cos(\pi\al)\G(2\al)}
\cdot\left(1+o(1)\right)\q {\rm as} \q t\to\f.
\eeq
\end{lem}

\begin{rem}
{\rm Taking into account the reflection formula
$\G(2\al)\G(1-2\al)=\pi/\sin(2\pi\al),$
the asymptotic relation (\ref{s4}) can be written in the following form}
\beq
\label{s04}
r(t)=C t^{2\al-1}\sin(\pi \al)\G(1-2\al)
\cdot\left(1+o(1)\right)\q {\rm as} \q t\to\f.
\eeq
\end{rem}

\n {\it Proof of Theorem \ref{M5}.\/}
By Lemma \ref{lem72} 1) and Parseval-Plancherel theorem we have
\beq
\label{s7}
S_{T,1}=\int_{-T}^T\left(1-\frac{|t|}T\right) r_1(t)r_2(t)\,dt=
2\pi\itf f_1(\la)f_2(\la)\,d\la - I_1-I_2,
\eeq
where
\beq
\label{s8}
I_1=\int_{|t|>T} r_1(t)r_2(t)\,dt
\eeq
and
\beq
\label{s9}
I_2=\frac1T\int_{-T}^T |t|r_1(t)r_2(t)\,dt.
\eeq
Hence, for $\al:\,=\al_1+\al_2<1/2$ from Lemma \ref{sl1} and (\ref{s8}),
we have as $T\to\f$
\bea
\label{s10}
\nonumber
I_1&=&\int_{|t|>T} r_1(t)r_2(t)\,dt
=2\int_{t>T} r_1(t)r_2(t)\,dt\\
\nonumber
&=&\frac{2C_1C_2\pi^2}{\cos(\pi\al_1)\cos(\pi\al_2)\G(2\al_1)\G(2\al_2)}
\cd \int_{t>T}t^{2(\al-1)}\left(1+o(1)\right)\\
&=&\frac{2C_1C_2\pi^2}{\cos(\pi\al_1)\cos(\pi\al_2)\G(2\al_1)\G(2\al_2)}
\cd\frac1{1-2\al}\cd T^{2\al-1}\left(1+o(1)\right).
\eea
For $I_2$, from Lemma \ref{sl1} and (\ref{s9}),
we have as $T\to\f$ and $\al:\,=\al_1+\al_2<1/2$
\bea
\label{s11}
\nonumber
I_2&=&\frac1T\int_{-T}^T |t|r_1(t)r_2(t)\,dt
=\frac2T\int_{0}^T t\, r_1(t)r_2(t)\,dt\\
\nonumber
&=&\frac2T\cd\frac{C_1C_2\pi^2}{\cos(\pi\al_1)\cos(\pi\al_2)\G(2\al_1)\G(2\al_2)}
\int_{0}^Tt^{2\al-1}\,dt\left(1+o(1)\right)\\
&=&\frac{2 C_1C_2\pi^2}{\cos(\pi\al_1)\cos(\pi\al_2)\G(2\al_1)\G(2\al_2)}
\cd\frac1{2\al}\cd T^{2\al-1}\left(1+o(1)\right).
\eea
From (\ref{s7}), (\ref{s10}) and (\ref{s11}) the result follows.
Theorem \ref{M5} is proved.

\begin{rem}
\label{FMC-rem}
\begin{itemize}
{\rm
\item[(a)]
As in Remark \ref{FM-rem}, the rate in (\ref{s2})
can be obtained from Theorem \ref{T5-3}.

\item[(b)]
We analyze the behavior of the approximations as
$\al_1, \al_2\to 1/4$.
First observe that the first-order asymptotic formula has
a pole when $\al:=\al_1+\al_2=1/2$. In particular,
denoting $\be:=\be_1+\be_2$, and using the change of variable
$\la^2=u$, we have
\bea
\label{s12}
\nonumber
2\pi\itf f_1(\la)f_2(\la)\,d\la
&=&
2\pi C_1C_2\itf  \frac 1 {|\la|^{2\al}(1+\la^2)^{\be}}d\la\\
&=& 2\pi C_1C_2 \int_0^\f \frac{u^{-1/2-\al}}{(1+u)^{\be}}du.
\eea
\n Applying the formula (see, e.g., \cite{Dv})
\bea
\label{s14}
\int_0^\f
\frac{u^{m-1}}{(1+u)^{m+n}}du=\frac{\G(m)\G(n)}{\G(m+n)},
\q m>0, \,\, n>0,
\eea
with $m=1/2-\al$ and $n=\be-m=\al+\be-1/2$, from (\ref{s12}) we find
\bea
\label{s15}
2\pi C_1C_2\itf f_1(\la)f_2(\la)\,d\la=
2\pi C_1C_2\frac{\G(1/2-\al)\G(\al+\be-1/2)}{\G(\be)}.
\eea
Then, using the Laurent expansion of the gamma function
$\G(1/2-\al)$ around the pole $\al=1/2$, from (\ref{s15}) we obtain: 
as $\al=\al_1+\al_2\to1/2$
\bea
\label{s16}
2\pi\itf f_1(\la)f_2(\la)\,d\la
= \frac{4 \pi C_1C_2}{1-2\al}+O(1).
\eea
However, the asymptotic behavior of the second-order term
$C(\al_1,\al_2)$ (see (\ref{s3})), as $\al_1,\al_2\to1/4$
 is readily seen to be
\bea
\label{s17}
\nonumber
C(\al_1,\al_2)&=&
 \frac{2 \pi^2 C_1C_2}{\cos(\pi\al_1)\cos(\pi\al_2)\G(2\al_1)\G(2\al_2)}
\cdot\frac{1}{2\al(1-2\al)}\\
\nonumber
&=& \frac{2 \pi^2 C_1C_2}{\cos^2(\pi/4)[\G(1/2)]^2}
\cdot\frac{1}{1-2\al}+O(1)\\
&=&\frac{4 \pi C_1C_2}{1-2\al}+O(1).
\eea
Thus, the pole in the first-order approximation is removed by the
second-order approximation, so that the approximation
\bea
\label{s18}
2\pi\itf f_1(\la)f_2(\la)\,d\la -
\frac{C(\al_1,\al_2)}{T^{1-2(\al_1+\al_2)}}
\eea
is bounded as $\al_1,\al_2\to1/4$.

\n This good behavior explains why the second-order
approximation produces a good approximation that does uniformly
well over $\al_1,\al_2\in[0,1/4]$, including the limits of the domain.

\item[(c)]
The second-order equivalence holds
along an arbitrary ray for which $\al=\al_1+\al_2\to 1/2$.
\n Indeed, let $\al_1^0\in[0,1/2]$ be any fixed number such that
$\al_1\to\al_1^0$ and $\al=\al_1+\al_2\to 1/2$, then
the representation (\ref{s16}) for the first-order asymptotic term
continues to apply.

\n On the other hand, using the reflection formula
$\G(z)\G(1-z)=\pi/\sin(\pi z)$ with $z=2\al_1^0$, we have
from (\ref{s3}) as $\al_1\to\al_1^0$ and $\al=\al_1+\al_2\to 1/2$
$$
C(\al_1,\al_2)=\frac{2\pi^2 C_1C_2}{\cos(\al_1\pi)\cos(\al_2\pi)\G(2\al_1)\G(2\al_2)}
\cdot\frac{1}{2\al(1-2\al)}
$$
\bea
\label{s20}
\nonumber
&&=\frac{2\pi^2 C_1C_2}{\cos(\al_1^0\pi)
\cos[(1-2\al_1^0)\pi/2]\G(2\al_1^0)\G(1-2\al_1^0)}
\cdot\frac{1}{1-2\al}+O(1)\\
\nonumber
&&=\frac{2\pi^2C_1C_2}{\cos(\al_1^0\pi)\sin(\al_1^0\pi)}
\cd\frac{\sin(2\al_1^0\pi)}{\pi}
\cdot\frac{1}{1-2\al}+O(1)\\
\nonumber
&&=\frac{4\pi C_1C_2}{1-2\al}+O(1),
\eea
and again the second-order equivalence holds.

\n Thus, in this special case, the second-order asymptotic
expansion removes the singularity in the first-order approximation,
and provides a substantially improved approximation to the original
functional.}
\end{itemize}
\end{rem}

\subsection{Limit Theorems for Toeplitz Quadratic Functionals}
\label{CLT}

In this section we examine the limit behavior of quadratic forms
and functionals of discrete- and continuous-time stationary Gaussian
processes with possibly long-range dependence.
The matrix and the operator that characterize the quadratic form and
functional are Toeplitz.

Let $\{X(u), \ u\in \mathbb{U}\}$ be a centered real-valued
Gaussian stationary process with spectral density
$f(\la)$, $\la\in \Lambda$ and covariance function
$r(t):=\widehat f(t)$, $t\in \mathbb{U}$, where $\mathbb{U}$
and $\Lambda$ are as in Section \ref{model}.
\n We are interested in the asymptotic distribution
(as $T\to\f$) of the following Toeplitz type quadratic
functionals of the process $X(u)$:
\beq
\label{c-1}
Q_T:=
 \left \{
 \begin{array}{ll}
\int_0^T\int_0^T\widehat g(t-s)X(t)X(s)\,dt\,ds&  
\mbox{in the continuous-time case}\\
\\
\sum_{k=1}^T\sum_{j=1}^T\widehat g(k-j)X(k)X(j) &
\mbox{in the discrete-time case},
\end{array}
\right.
\eeq
where
\begin{equation}
\label{c-2}
%\label{1-3}
\widehat g(t)=\int_\Lambda e^{i\la t}\,g(\la)\,d\la,\,\,t\in \mathbb{U}
\end{equation}
is the Fourier transform of some real, even, integrable
function $g(\la),$ $\la\in\Lambda$.
We will refer $g(\la)$ as a
generating function for the functional $Q_T$.
In the discrete-time case the functions $f(\la)$ and
$g(\la)$ are assumed to be $2\pi$-periodic and periodically
extended to $\mathbb{R}$.

The limit distributions of the functionals (\ref{c-1}) are
completely determined by the spectral density $f(\la)$ and the
generating function $g(\la)$, and depending on their properties
the limit distributions can be either Gaussian
(i.e., $Q_T$ with an appropriate normalization obeys central
limit theorem), or non-Gaussian.
The following two questions arise naturally:

\vskip1mm

a) Under what conditions on $f(\la)$ and $g(\la)$ will
the limits be Gaussian?

b) Describe the limit distributions, if they are non-Gaussian.

\subsubsection{Central limit theorems for Toeplitz quadratic functionals}
\label{CLT-1}
We first discuss the question a), that is,
finding conditions on the spectral density $f(\la)$ and the
generating function $g(\la)$ under which the functional $Q_T$,
defined by (\ref{c-1}), obeys central limit theorem.

\n This question goes back to the classical monograph by Grenander and
Szeg\"o \cite{GS}, where the problem was considered for discrete
time processes, as an application of the authors' theory of the
asymptotic behavior of the trace of products of truncated Toeplitz
matrices.

Later the problem a) was studied by Ibragimov \cite{I} and
Rosenblatt \cite{R2}, in connection with the statistical estimation of
the spectral ($F(\la)$) and covariance ($r(t)$) functions,
respectively. Since 1986, there has been a renewed interest in
both questions a) and b), related to the statistical inferences
for long-memory processes (see, e.g., Avram \cite{A}, Fox
and Taqqu \cite{FT1}, Giraitis and Surgailis \cite{GSu},
Giraitis et al. \cite{GKSu}, Terrin
and Taqqu \cite{TT3}, Taniguchi \cite{Ta1}, Taniguchi  and
Kakizawa \cite{TK}, Ginovyan and Sahakyan \cite{GS1}, and
references therein). In particular, Avram \cite{A}, Fox and Taqqu
\cite{FT1}, Giraitis and Surgailis \cite{GSu}, Ginovyan and
Sahakyan \cite{GS1} have obtained sufficient conditions for
the quadratic form $Q_T$ to obey the central limit theorem (CLT),
when the model $X(t)$ is a discrete-time process.

For continuous time processes the question a) was studied in
Ibragimov \cite{I}, Ginovyan \cite{G1,G3},
and Ginovyan and Sahakyan \cite{GS2}.

Let $Q_T$ be as in (\ref{c-1}). We will use the following notation:
By $\widetilde Q_T$ we denote the standard normalized
quadratic functional:
\begin{equation}
\label{c-3}
%\label{1-4}
 \widetilde Q_T=\frac1{\sqrt T}\,\left(Q_T-EQ_T\right).
\end{equation}
The notation
\begin{equation}
\label{c-4}
%\label{1-5}
\widetilde Q_T\an N(0,\sigma^2) \q {\rm as}\q T\to\f
\end{equation}
will mean that the distribution of the random variable
$\widetilde Q_T$ tends (as $T\to\f$) to the centered normal
distribution with variance $\sigma^2$.

Our study of the asymptotic distribution of the quadratic
functionals (\ref{c-1}) is based on the following representation
of the $k$--th order cumulant $\chi_k(\cd)$ of $\widetilde Q_T$,
which follows from (\ref{MTc-5}) (see, also, \cite{GS,I}):
\beq
\label{c-5}
%\label{1-6}
\chi_k(\widetilde Q_T)= \left \{
\begin{array}{ll}
0, & \mbox{for $k = 1$}\\
\\[-1mm]
 T^{-k/2}2^{k-1}(k-1)! \,
\text{tr} \, [A_T(f)A_T(g)]^k, & \mbox{for $k \ge 2$,}
\end{array}
\right.
\eeq
where
$A_T(f)$ and $A_T(g)$ denote either the $T$-truncated Toeplitz
operators (for continuous-time case), or the $T\times T$ Toeplitz matrices
(for discrete-time case) generated by the functions $f$ and $g$ respectively,
and  $\tr[A]$ stands for the trace of an operator $A$.

The next result contains sufficient conditions in terms
of $f(\la)$ and $g(\la)$ ensuring central limit theorems
for standard normalized quadratic functionals $\widetilde Q_T$ both
for discrete- and continuous-time processes.

Below we assume that $f, g \in L^1(\Lambda)$, and with no
loss of generality, that $g\ge 0$.   Also, we set
\beq
\label{c-6}
%\label{1-8}
\si^2_0 = 16\pi^3\int_\Lambda f^2(\la)g^2(\la)\,d\la.
\eeq
As usual $\Lambda=\mathbb{T}=(-\pi, \pi]$ or $\Lambda=\mathbb{R}=(-\f,\f)$. 
The following theorem includes
both discrete-time ($\Lambda=(-\pi, \pi]$) and
continuous-time ($\Lambda=(-\f,\f)$).
\begin{thm}
\label{cth1}
Each of the following conditions is sufficient for
\beq
\label{c-7}
\widetilde Q_T\an N(0,\sigma_0^2) \q {\rm as}\q T\to\f,
\eeq
with $\si^2_0$ given by (\ref{c-6}).
\begin{itemize}
\item[{\bf (A)}]
$f\cdot g\in L^1(\Lambda)\cap L^2(\Lambda)$ and
\beq
\label{c-8}
\chi_2(\widetilde Q_T):=
\frac2T\tr\bigl[B_T(f)B_T(g)\bigr]^2 \longrightarrow
\sigma_0^2<\f.
\eeq
\item[{\bf (B)}]
The function
\beq
\label{c-9}
\varphi({\bf u}):=\varphi(u_1, u_2,u_3)=\int_\Lambda
f(\la)g(\la-u_1)f(\la-u_2)g(\la-u_3)\,d\la
\eeq
belongs to $L^2(\Lambda^3)$ and is continuous at ${\bf 0}=(0,0,0).$

\item[{\bf (C)}]
$f \in L^1(\Lambda)\cap L^p(\Lambda)$ $(p\ge2)$
and  $g\in L^1(\Lambda)\cap L^q(\Lambda)$ $(q\ge2)$
with $$1/p+1/q\le1/2.$$

\item[{\bf (D)}]
$f\in L^1(\Lambda)\cap L^2(\Lambda)$,
\,$g\in L^1(\Lambda)\cap L^2(\Lambda)$,
$fg\in L^2(\Lambda)$ and
\beq
\label{c-10}
\nonumber
\int_\Lambda  f^2(\la)g^2(\la-\mu)\,d\la \longrightarrow
 \int_\Lambda f^2(\la)g^2(\la)\,d\la \quad {\rm as} \quad \mu\to0.
\eeq

\item[{\bf (E)}]
The spectral density $f(\la)$ and the generating function $g(\la)$ satisfy
\beq
\label{c-11}
\nonumber
f(\lambda)\le |\lambda|^{-\alpha}L_1(\lambda), \q
|g(\lambda)|\le |\lambda|^{-\beta}L_2(\lambda),\q \lambda\in\Lambda,
\eeq
for some $\alpha<1,\ \beta<1$ with
$$\alpha+\beta\le1/2 \q {\rm and} \q L_i\in SV(\mathbb{R}),
\,\, \lambda^{-(\alpha+\beta)}L_i(\lambda)\in L^2(\Lambda),\ \ i=1,2,$$
where $SV(\mathbb{R})$ is the class of slowly varying
at zero functions $u(\lambda)$, $\lambda\in\mathbb{R}$, satisfying
$u(\lambda)\in L^\infty(\mathbb{R}),$\
$\lim_{\lambda\to0}u(\lambda)=0,$ \
$u(\lambda)=u(-\lambda)$ and $0<u(\lambda)<u(\mu)$\ for\ $0<\lambda<\mu.$

In the continuous-time case, we also assume that the functions $f(\la)$ and $g(\la)$ are bounded on
$\mathbb{R}\setminus (-\pi, \pi)$.
\end{itemize}
\end{thm}

\begin{rem}
{\rm For discrete-time case:
assertions (A) and (D) were proved in Giraitis and Surgailis \cite{GSu}
(see also Giraitis et al. \cite{GKSu});
assertions (B) and (E) were proved in Ginovyan and Sahakyan \cite{GS1};
assertion (E) with $\al+\be<1/2$ was first obtained by Fox and Taqqu \cite{FT1};
assertion (C) for $p=q=\f$ was first established by Grenander and Szeg\"o
(\cite{GS}, Sec. 11.7),
while the case $p=2$, $q=\f$ was proved by Ibragimov
\cite{I} and Rosenblatt \cite{R2}, in the general discrete-time case
assertion (D) was proved by Avram \cite{A}.

For continuous-time case assertions (A) -- (E) were proved in Ginovyan \cite{G3}
and Ginovyan and Sahakyan \cite{GS2}.}
\end{rem}

\begin{rem}
{\rm Assertion (A) implies assertions (B) -- (E).
Assertion (B) implies assertions (C) and (D).
On the other hand, for functions $f(\la)=\la^{-3/4}$ and $g(\la)=\la^{3/4}$
satisfying the conditions of (E), the function
$\I(t_1, t_2,t_3)$ is not defined for \,$t_2=0$, $t_1\neq 0$, $t_3\neq 0$,
showing that assertion (B) generally does not imply assertion (E)
(see Ginovyan and Sahakyan \cite{GS1}).}
\end{rem}

\begin{rem}
{\rm Examples of spectral density $f(\la)$ and
generating function $g(\la)$ satisfying the conditions
of Theorem \ref{cth1} (E) provide the functions
\beq
\label{c-13}
f(\la)= |\la|^{-\al}|\ln|\la||^{-\g} \q \mbox{and} \q
g(\la)= |\la|^{-\be}|\ln|\la||^{-\g},
\eeq
where $\al<1,$\ $\be<1,$\ $\al+\be\le1/2$ and $\g>1/2$
(see Ginovyan and Sahakyan \cite{GS1}).}
\end{rem}
\begin{rem}
{\rm The functions $f(\la)$ and $g(\la)$ in Theorem \ref{cth1} (E)
have singularities at the point $\la=0$, and are bounded in
any neighborhood of this point. It can be shown that the choice
of the point $\la=0$ is not essential, and instead
it can be taken to be any point $\la_0\in[-\pi, \pi]$.
Using the properties of the products of Toeplitz matrices
and operators it can be shown that Theorem \ref{cth1} (E)
remains valid when $f(\la)$ and $g(\la)$ have
singularities of the form (\ref{c-11}) at the
same finite number of points of the segment $[-\pi, \pi]$
(see Ginovyan and Sahakyan \cite{GS1}).}
\end{rem}

\begin{rem}
{\rm
The next proposition shows that the condition of positiveness
and finiteness of asymptotic variance of quadratic form $Q_T$
is not sufficient for $Q_T$ to obey CLT as it was conjectured
in Giraitis and Surgailis \cite{GSu}, and Ginovyan \cite{G2}.}
\end{rem}
\begin{pp}
There exist a spectral density $f(\la)$ and a generating function
$g(\la)$ such that
\beq
\label{c-14}
0<\inl f^2(\la)\,g^2(\la)\,d\la<\f
\eeq
and
\beq
\label{c-15}
\lim_{T\to \infty}\sup\chi_2(\widetilde Q_T)
=\lim_{n\to \infty}\sup\frac2T\tr\left(B_T(f)B_T(g)\right)^2=\infty,
\eeq
that is, the condition (\ref{c-14}) does not guarantee convergence
in (\ref{c-8}).
\end{pp}
To construct functions $f(\la)$ and $g(\la)$ satisfying
(\ref{c-14}) and (\ref{c-15}), for a fixed $p\ge2$
we choose a number $q>1$ to satisfy $1/p+1/q>1$,
and for such $p$ and $q$ we consider the functions
$f_0(\la)$ and $g_0(\la)$ defined by
\beq
\label{c-16}
f_0(\la)=
 \left \{
 \begin{array}{ll}
\left(\frac{2^s}{s^2}\right)^{1/p}, & \mbox{if $2^{-s-1}\le
\la \le 2^{-s},\,s=2m$}\\
0,&  \mbox{if $2^{-s-1}\le\la \le 2^{-s},\,s=2m+1$}
\end{array}
\right.
\eeq
\beq
\label{c-17}
g_0(\la)=
 \left \{
\begin{array}{ll}
 \left(\frac{2^s}{s^2}\right)^{1/q}, & \mbox{if $2^{-s-1}\le
\la \le 2^{-s},\,s=2m+1$}\\
0,& \mbox{if $2^{-s-1}\le\la \le
2^{-s},\,s=2m$},\end{array} \right.
\eeq
where $m$ is a positive integer. For an
arbitrary finite positive constant $C$ we set
$g_\pm(\la)=g_0(\la)\pm C$.
Then the functions $f=f_0$ and $g=g_+$ or $g=g_-$
satisfy (\ref{c-14}) and (\ref{c-15})
(for details we refer to Ginovyan and Sahakyan \cite{GS1}.
Consequently, for these functions
the standard normalized quadratic form $Q_T$ does not obey CLT,
and it is of interest to describe the limiting non-Gaussian
distribution of $Q_T$ in this special case.

\subsubsection{Non-central Limit Theorems}
\label{CLT-2}

The problem b) for discrete-time processes, that is,
the description of the limit distributions of the
quadratic form
\beq
\label{nc-1}
Q_T:=\sum_{k=1}^T\sum_{j=1}^T\widehat g(k-j)X(k)X(j),
\q T\in\mathbb{N}
\eeq
if it is non-Gaussian, goes back to the papers
by Rosenblatt \cite{R1}-\cite{R3}.

Later this problem was studied in a series of papers
by Taqqu, and Terrin and Taqqu (see, e.g., in \cite{Tq1},
\cite{Tq4}, \cite{TT1}, \cite{TT3}, and references therein).
Specifically, suppose that the spectral density $f(\la)$ and
the generating function $g(\la)$ are regularly varying functions
at the origin:
\beq
\label{nc-2}
f(\la)= |\la|^{-\al}L_1(\la) \q {\rm and} \q
g(\la)= |\la|^{-\be}L_2(\la), \q \al<1, \be<1,
\eeq
where $L_1(\la)$ and $L_2(\la)$ are slowly varying functions
at zero, which are bounded on bounded intervals.
The conditions $\al<1$ and $\be<1$ ensure that the Fourier
coefficients of $f$ and $g$ are well defined. When $\al>0$
the model $\{X(t), t\in\mathbb{Z}\}$ exhibits long memory.

It is the sum $\al+\be$ that determines the asymptotic behavior
of the quadratic form $Q_T$. If $\al+\be\le 1/2$, then by
Theorem \ref{cth1}(E) the standard normalized quadratic form
$$T^{-1/2}\left(Q_T-EQ_T\right)$$
converges in distribution to a Gaussian random variable.
If $\al+\be>1/2$, convergence to Gaussian fails.

Consider the embedding of the discrete sequence
$\{Q_T, \,T\in\mathbb{N}\}$ into a conti\-nuous-time process
$\{Q_T(t), \,T\in\mathbb{N}, t\in\mathbb{R}\}$ defined by
\beq
\label{nc-3}
Q_T(t):=\sum_{k=1}^{[Tt]}\sum_{j=1}^{[Tt]}\widehat g(k-j)X(k)X(j),
\eeq
where $[\,\,]$ stands for the greatest integer.
Denote by $Z(\cd)$ the complex-valued Gaussian random measure defined
on the Borel $\si$-algebra $\mathcal{B}(\mathbb{R})$, and satisfying
$EZ(B)=0$, $E|Z(B)|^2=|B|$, and $\ol{Z(-B)}=Z(B)$ for any
$B\in \mathcal{B}(\mathbb{R})$.

The next result, proved in Terrin and Taqqu \cite{TT1},
describes the non-Gaussian limit distribution
of the suitable normalized process $Q_T(t)$.
\begin{thm}
\label{NCT}
Let $f(\la)$ and $g(\la)$ be as in (\ref{nc-2}) with $\al<1$, $\be<1$
and slowly varying at zero and bounded on bounded intervals factors
$L_1(\la)$ and $L_2(\la)$. Let the process $Q_T(t)$ be as in (\ref{nc-3}).
Then for $\al+\be>1/2$
\begin{equation}
\label{1-16}
 \widehat Q_T(t):=\frac1{T^{\al+\be}L_1(1/T)L_2(1/T)}\,
 \left(Q_T(t)-E[Q_T(t)]\right)
\end{equation}
converges (as $T\to\f$) weakly in $D[0,1]$ to
\beq
\label{nc-4}
Q(t):=\int_{\mathbb{R}^2}^{''}K_t(x,y)dZ(x)dZ(y),
\eeq
where
\beq
\label{nc-5}
K_t(x,y)=|xy|^{-\al/2}\int_\mathbb{R}\frac{e^{it(x+u)}-1}{i(x+u)}
\cd \frac{e^{it(y-u)}-1}{i(y-u)}|u|^{-\be}du,
\eeq
The double prime in the integral (\ref{nc-4}) indicates that the
integration excludes the diagonals $x=\pm y$.
\end{thm}

\begin{rem}
{\rm The limiting process in (\ref{nc-4}) is real-valued, non-Gaussian,
and satisfies $EQ(t)=0$ and $EQ^2(t)=\int_{\mathbb{R}^2}|K_t(x,y)|^2dxdy$.
It is self-similar with parameter $H=\al+\be\in(1/2,2)$, that is, the
processes $\{Q(at), t\ge0\}$ and $\{a^HQ(t), t\ge0\}$ have the same
finite dimensional distributions for all $a>0$.}
\end{rem}

\begin{rem}
{\rm In \cite{R1} (see also \cite{R3}) Rosenblatt showed that if a
discrete-time centered Gaussian process $X(t)$ has covariance function
$r(t)=(1+t^2)^{\al/2-1/2}$ with $1/2<\al<1$, then the random variable $$Q_T:=T^{-\al}\sum_{k=1}^T\left[X^2(k)-1\right]$$
has
a non-Gaussian limiting distribution, and described this distribution
in terms of characteristic functions. This is a special case of
Theorem \ref{NCT} with $t=1$, $1/2<\al<1$ and $\be=0$.
In \cite{Tq1} (see also \cite{Tq4}) Taqqu extended Rosenblatt's result
by showing that the stochastic process
$$Q_T(t):=T^{-\al}\sum_{k=1}^{[Tt]}\left[X^2(k)-1\right]$$
converges (as $T\to\f$) weakly to a process (called Rosenblatt process)
which has the double Wiener-It\^o integral representation
\beq
\label{nc-6}
Q(t):=C_\al\int_{\mathbb{R}^2}^{''}\frac{e^{it(x+y)}-1}{i(x+y)}
|x|^{-\al/2}|y|^{-\al/2}dZ(x)dZ(y).
\eeq
The distribution of the random variable $Q(t)$ in (\ref{nc-6}) for
$t =1$ is described in Veillette and Taqqu \cite{VT}.
}
\end{rem}

\begin{rem}
{\rm The slowly varying functions $L_1$ and $L_2$ in (\ref{nc-2})
are of importance because they provide a great flexibility
in the choice of spectral density $f$ and generating function $g$.
Observe that in Theorem \ref{NCT} the functions $L_1$ and $L_2$
influence only the normalization (see (\ref{1-16})),
but not the limit $Q(t)$.
Theorem \ref{cth1}(E) shows that in the critical case $\al+\be=1/2$
the limit distribution of the standard normalized quadratic form $Q_T$
is Gaussian and essentially depends on slowly varying factors
$L_1$ and $L_2$.

Note also that the critical case $\al+\be=1/2$ was partially
investigated by Terrin and Taqqu in \cite{TT3}.
Starting from the limiting random variable $Q(1)= Q(1;\al, \be)$,
which exists only when $\al+\be>1/2$, they showed the random variable
$$(\al+\be-1/2)Q(1;\al, \be)$$
converges in distribution to a Gaussian
random variable as $\al+\be$ approaches to $1/2$.}
\end{rem}

\begin{rem}
{\rm For continuous-time processes the problem b) has not been investigated, 
and it would be of interest to describe the limiting non-Gaussian distribution 
of the quadratic functional $Q_T$.}
\end{rem}

\subsection{Berry-Ess\'een Bounds and Large Deviations for
Toeplitz Quadratic Functionals}

In this section, we briefly discuss Berry-Ess\'een bounds in CLT
and large deviations principle for quadratic functionals
both for continuous- and discrete-time Gaussian stationary processes
(for more about these topics we refer to
\cite{BGR1,BrD,DC,GRZ,Ih,NP,Ta1}, and reference therein).

\vskip1mm
\n {\bf Berry-Ess\'een Bounds.}
Let $Q_T$ and $\widetilde Q_T$
be as in (\ref{c-1}) and (\ref{c-3}), respectively.
Denote $\widehat Q_T:=\widetilde Q_T/\sqrt{{\rm Var}(\widetilde Q_T)}$,
and let $Z$ be the standard normal random variable: $Z \sim N(0,1)$.
The CLT for $Q_T$  (Theorem \ref{cth1}) tells us that
$\widehat Q_T\longrightarrow Z$ in distribution as $T\to\f$.
The natural next step concerns the closeness between the
distribution of $\widehat Q_T$ and standard normal distribution,
which means asking for the rate of convergence in the CLT.
Results of this type are known as Berry-Ess\'een bounds (or asymptotics).

In discrete-time case, for special
quadratic functionals, Berry-Ess\'een bounds were established
in  Tanoguchi \cite {Ta1},
while for continuous-time case, Berry-Ess\'een-type bounds were obtained in
Nourdin and Peccati \cite {NP}.
The next theorem captures both cases.
\begin{thm}
\label{bth1}
Let $\widetilde Q_T$ be as in (\ref{c-3}),
$\widehat Q_T:=\widetilde Q_T/\sqrt{{\rm Var}(\widetilde Q_T)}$,
and $\Phi(z)=P(Z\le z)$, where $Z \sim N(0,1)$.
Assume that $f(\la)\in L^1(\Lambda)\cap L^p(\Lambda)$ $(p\ge1)$
and  $g(\la)\in L^1(\Lambda)\cap L^q(\Lambda)$ $(q\ge1)$.
The following assertions hold.
\begin{itemize}
\item[1.]
If $1/p+1/q\le1/4$, then there exists a constant
$C=C(f,g)>0$ such that for all $T>0$ we have
\beq
\label{be1}
\sup_{z\in\mathbb{R}}|P(\widehat Q_T\le z)-\Phi(z)|
\le\frac{C}{\sqrt{T}}.
\eeq

\item[2.]
If $1/p+1/q\le1/8$ and $ \int_\Lambda f^3(\la)g^3(\la)\,d\la \ne0$,
then there exist a constant $c=c(f,g)>0$ and a number
$T_0=T_0(f,g)>0$ such that $T>T_0$ implies
\beq
\label{be2}
\sup_{z\in\mathbb{R}}|P(\widehat Q_T\le z)-\Phi(z)|
\ge\frac{c}{\sqrt{T}}.
\eeq
More precisely, for any $z\in\mathbb{R}$, we have as $T\to\f$
\beq
\label{be3}
\sqrt{T}|P(\widehat Q_T\le z)-\Phi(z)|\longrightarrow
\sqrt{\frac23}\frac{\int_\Lambda f^3(\la)g^3(\la)d\la}
{\left(\int_\Lambda f^2(\la)g^2(\la)d\la\right)^{3/2}}(1-z^2)e^{-z^2/2}.
\eeq
\end{itemize}
\end{thm}

\begin{rem}
{\rm In the continuous-time case, Theorem \ref{bth1} was proved
in Nourdin and Peccati \cite {NP}, by appealing to
a general CLT of Section \ref{CLT} (Theorem \ref{cth1}),
and Stein's method. The proof, in the discrete-time case,
is similar to that of the continuous-time case.}
\end{rem}

\n {\bf Large Deviations.}
We now present sufficient conditions that ensure large deviations
principle (LDP) for Toeplitz type quadratic functionals of
stationary Gaussian processes.
For more about LDP we refer to Bryc and  Dembo \cite{BrD},
Bercu et al. \cite{BGR1}, Sato et al. \cite{SKT},
Taniguchi and Kakizawa \cite{TK}, and references therein.

First observe that large deviation theory can be viewed
as an extension of the law of large numbers (LLN).
The LLN states that certain probabilities converge to zero,
while the large deviation theory focuses on the rate of convergence.
Specifically, consider a sequence of random variables $\{\xi_n, \, n\ge1\}$
converging in probability to a real constant $m$.
Note that $\xi_n$ could represent, for instance, the $n$-th partial
sum of another sequence of random variables:
$\xi_n = \frac1n \sum_{k=1}^n\eta_k$,
where the sequence $\{\eta_k\}$ may be independent identically distributed,
or dependent as in an observed stretch of a stochastic process.
By the LLN, we have for $\vs>0$
\beq
\label{ld1}
\Ph\{|\xi_n - m|>\vs\} \to 0 \q {\rm as} \q n\to\f.
\eeq
It is often the case that the convergence in (\ref{ld1})
is exponentially fast, that is,
\beq
\label{ld2}
\Ph\{|\xi_n - m|>\vs\} \approx R(\cd)\exp[-nI(\vs,m)] \q {\rm as} \q n\to\f,
\eeq
where $R(\cd) = R(\vs,m,n)$ is a slowly varying
(relative to an exponential) function of $n$
and $I(\vs,m)$ is a positive quantity.
\n Loosely, if (\ref{ld2}) holds, we say that the sequence $\{\xi_n\}$
satisfies a large deviations principle.
One of the basic problems of the large deviation theory is
to determine $I(\vs,m)$ and $R(\vs,m,n).$
To be more precise, we recall the definition of
Large Deviation Principle (LDP) (see, for instance,  \cite{BrD,TK}).

\begin{den}
\label{ldd1}
{\rm Let  $\{\xi_n,$  $n \in \mathbb{Z}\}$  be a sequence of real--valued
random variables defined on the probability space $(\Om, {\cal F}, \Ph)$.
We say that $\{\xi_n \}$ satisfies a Large Deviation Principle (LDP)
with speed $a_n \to 0$ and rate function $I: \mathbb{R}\to [0,\infty]$,
if $I(x)$ is lower semicontinuous,
that is, if $x_n\to x$ then $\liminf_{n\to\f}I(x_n)\ge I(x)$, and
$$\liminf_{n\to\infty}a_n \log \Ph\{\xi_n\in A\}\geq -
\inf_{x\in A} I(x)$$
for all open subsets $A\subset \mathbb{R}$, while
$$\limsup_{n\to\infty} a_n \log \Ph\{\xi_n\in B\}\leq -
\inf_{x\in B} I(x)$$
for all closed subsets $B\subset \mathbb{R}$.
The function
$I(x)$ is called a good rate function if its level sets are compact,
that is, the set $\{x\in \mathbb{R}: I(x)\le b\}$ is compact
for each $b\in \mathbb{R}$}.
\end{den}

Now let $Q_T$ be the Toeplitz type quadratic
functionals of a process $X(u)$ defined by
(\ref{c-1}) with spectral density $f(\la)$ and generating function
$g(\la)$.

The next result states sufficient conditions in terms
of $f(\la)$ and $g(\la)$ to ensure that the LDP
for normalized quadratic functionals $\{\frac1TQ_T\}$
holds both for discrete- and continuous-time processes.

\begin{thm}
\label{4-10}
Assume that $f(\la)g(\la) \in L^\f(\Lambda)$.
Then the random variable $\{\frac1TQ_T\}$ satisfies a LDP
with the speed $a_T = \frac1T$ and a good rate function $I(x)$:
$$
I(x)=\sup_{-\infty<y<1/(2C)} \{xy-V(f,g;y)\},
$$
where $C=\hbox{ess sup} f(\la)|g(\la)|$ and for $y<1/(2C)$
$$
V(f,g;y) =-\frac{1}{4\pi}\int_\Lambda \log (1- 2y f(\la)g(\la))d\la.
$$
\end{thm}

\begin{rem} {\rm
In the special case of $g(\la) = 1$, Theorem \ref{4-10} was proved by
Bryc and  Dembo \cite{BrD}, for general case we refer to
Bercu et al. \cite{BGR1}.}
\end{rem}

\section{Proof of Theorems \ref{T5} -  \ref{T5-3}}
\label{PR}

We only prove the results concerning Toeplitz operators
(Theorems \ref{T5} -  \ref{T5-3}). The proofs of the corresponding
results for Toeplitz matrices are similar.
First we state a number of technical lemmas, 
which are proved in the Appendix.

The following result is known (see, e.g., \cite{GS2}, \cite{GKSu}, p. 8).
\begin{lem}
\label{lem01}
Let $D_T(u)$ be the Dirichlet kernel
\begin{equation}
\label{a-2} D_T(u)=\frac{\sin(Tu/2)}{u/2}.
\end{equation}
Then, for any $\de\in(0,1)$
\beq
\label{01}
|D_T(u)|\leq 2 \,{T^\delta}{|u|^{\delta-1}},\quad u\in \rr.
\eeq
\end{lem}

Denote
\bea
\label{lem72-0}
&&G_T(u):=\int_0^Te^{iTu}dt=e^{iTu/2}D_T(u),\quad u\in\mathbb{R},\\
\label{lem72-1}
&& \Phi_T({\bf u}):\,
%=\Phi_T(u_1, \ldots, u_{2\nu-1})
=\frac1{(2\pi)^{m-1}T}\cd D_T(u_1)\cdots
D_T(u_{m-1})D_T(u_1+\cdots+u_{m-1}),\\
&&\label{lem72-2}
\Psi({\bf u}):\,
%=\Psi(u_1,\ldots,u_{2\nu-1})
=\I(u_1,u_1+u_2,\ldots, u_1+\cdots+u_{m-1}),
\eea
where ${\bf u}=(u_1,\ldots, u_{m-1})\in \rr^{m-1}$ and the function
$\varphi({\bf u})$, corresponding to the collection
$\mathcal{H}=\{h_1,h_2,\ldots,h_m\}$,
is defined by (\ref{n 4-6}).

The next lemma follows from  (\ref{n 4-4}) and
(\ref{lem72-0}) -- (\ref{lem72-2})
(cf. \cite{GS2}, Lemma 1).

\begin{lem}
\label{lem72}
Let  $\mathcal{H}=\{h_1,h_2,\ldots,h_m\}$ be a collection
of integrable real symmetric functions on $\mathbb{R}$,
${\hat h}_k$ be the Fourier transform of function $h_k$ $(k=1,\ldots, m)$,
and let $S(T):=S_{W,\mathcal{H}}(T)$ be as in (\ref{n 4-4}).
The following equalities hold.
\beaa
\label{lem72-6}
&&1)\ S(T)=\frac 1T\int_{0}^T\ldots \int_0^T
{\hat h}_1(u_1-u_2){\hat h}_2(u_2-u_3)\ldots {\hat h}_m(u_m-u_1)\\
\label{lem72-7}
&&2)\ S(T)=\frac 1T \int_{\rr^m}
h_1(u_1)\ldots h_m(u_m)G_T(u_1-u_2)G_T(u_2-u_3)\cdots \\
&&\hskip5cm\times G_T(u_m-u_1)du_1\ldots du_m.\\
&&3)\ S(T)= {(2\pi)^{m-1}}\int_{\rr^{m-1}}\Psi({\bf u})
 \Phi_T({\bf u})d{\bf u}.
 \label{lem72-8}
\eeaa
\end{lem}
For $m=3,4\ldots$ and $\delta>0$ we denote
$$
{\mathbb E}_\de=\{(u_1,\ldots, u_{m-1})\in {\mathbb R }^{m-1}:
\, |u_i|\le\de, \, i=1,\ldots,m-1\},\quad
{\mathbb E}_\de^c={\mathbb R}^{m-1}\sm{\mathbb E}_\de
$$ 
and
$$
p_3=2,\quad p(m)=\frac{m-2}{m-3}\quad (m>3).
$$
\begin{lem}\label{lem73}
The kernel  $\Phi_T(\uu),\ \uu\in {{\mathbb R}^{m-1}}$, $m\geq 3$
possesses the following properties:
\begin{align}
\label{z0}
a)&  \q \sup_{T}\int_{{\mathbb R}^{m-1}}|\Phi_T(\uu)|\,d\uu =C_1<\f;\notag\\
b)&  \q\int_{{\mathbb R}^{m-1}}\Phi_T(\uu)\,d\uu =1;\notag\\
c)& \q\lim_{T\to\f}\int_{{\mathbb E}^c_\de}|\Phi_T(\uu)|\,d\uu
=0 \q\text {for any} \q \de>0;\notag\\
d)&\q\text {for any $\de>0$ there exists a constant
$C_\de>0$ such that}\notag\\
& \q \int_{\mathbb{E}^c_\de}\left|\Phi_T(\uu)\right|^{p(m)}d\uu\le C_\de
\q \text {for \q $T>0$},
\end{align}
\end{lem}
The proof of the next lemma can be found in \cite{GKSu}, p. 161.
\begin{lem}
\label{lem1}
Let $0<\beta<1$, $0<\alpha<1$, and $\alpha+\beta>1$.
Then for any $y\in \rr,\ y\neq0$,
\begin{equation}
\label{a-4}
\int_\rr \frac1{|x|^\alpha|x+y|^\beta}dx = \frac C
{|y|^{\alpha+\beta-1}},
\end{equation}
 where $C$ is a constant depending on $\alpha$ and $\beta$.
\end{lem}

Denote $E=\{(u_1,u_2,\ldots,u_n)\in\rr^n: |u_i|\leq1,\ i=1,2,\ldots,n\}$
and let $E^C=\rr^n\setminus E$.

\begin{lem}
\label{lem2}
Let \ $0<\alpha\le 1$ and $\frac n{n+1}<\beta<\frac{n+\alpha}{n+1}$. Then
\begin{equation}
\label{a-05}
B_i:=\int_{E}
\frac{|u_i|^\alpha}{|u_1\cdots u_n(u_1+\cdots+u_n)|^\beta} \,
du_1\cdots du_n<\infty,
\q i=1,\ldots,n.
\end{equation}
\end{lem}

\begin{lem}
\label{lem3}
Let $\frac n{n+1}<\beta<1$. Then
\begin{equation}
\label{a-6}
I:=\int_{E^C}
\frac1{|u_1\cdots u_n(u_1+\cdots+u_n)|^\beta} \, du_1\cdots du_n<\infty.
\end{equation}
\end{lem}

\vskip2mm

\n {\sl Proof of Theorem \ref{T5}.} We start with {\bf (A2)}.
By Lemma \ref{lem72}, 3) and \ref{lem73}, b), we have
$\Delta(T)= (2\pi)^{m-1}|R(T)|$, where
$$
R(T)=
\int_{\mathbb{R}^3}[\Psi(\uu)-\Psi(\0)]\Phi_T(\uu)d\uu.
$$
It follows from (\ref{lem72-2}) that the function $\Psi(\uu)$ belongs to
$ L^{m-2}(\mathbb{R}^{m-1})$ and is continuous
at ${\bf 0}=(0,\ldots,0)\in \mathbb{R}^{m-1}$.
Hence for any  $\vs>0$ we can find  $\de>0$ to satisfy
\beq\label{d-3}
|\Psi(\uu)-\Psi(\0)|<\frac \vs{C_1},\quad \uu\in \mathbb{E}_\delta,
\eeq
where $C_1$ is the constant from Lemma \ref{lem73}, a).
Consider the decomposition $\Psi=\Psi_1+\Psi_2$ such that
\beq\label{d-4}
\|\Psi_1\|_{(m-2)}\le \frac \vs{\sqrt{C_\de}} \q\text {and}\q
\|\Psi_2\|_\infty<\infty,
\eeq
where ${C_\de}$ is as in Lemma \ref{lem73}, d).
%\eject

Observe that $\frac 1{m-2}+\frac1{p(m)}=1$, where $p(m)=\frac{m-2}{m-3}$.
Hence, applying Lemma \ref{lem73} and (\ref{d-3}), (\ref{d-4})
for sufficiently large $T$ we obtain
\beaa\label{d-5}
\notag
|R(T)|&\le&
\int_{\mathbb{E}_\de}|\Psi(\uu)-\Psi(\0)||\Phi_T(\uu)|d\uu
+C_m\int_{\mathbb{E}^c_\de}|\Psi_1(\uu)||\Phi_T(\uu)|d\uu\\
&+&\int_{\mathbb{E}^c_\de}|\Psi_2(\uu)-\Psi(\0)||\Phi_T(\uu)|d\uu
\le  \frac\vs{C_1}\int_{\mathbb{E}_\de}|\Phi_T(\uu)|d\uu\\
&+&\|\Psi_1\|_{(m-2)}\biggl[\int_{\mathbb{E}^c_\de}\Phi_T^{p(m)}
(\uu)d\uu\biggr]^{1/{p(m)}}
+C_2 \int_{\mathbb{E}^c_\de}|\Phi_T(\uu)|d\uu\le 3\, \vs,
\eeaa
and the result follows.
\qed

\vskip2mm
\n
{\sl Proof of {\bf (A1)}}.
According to Theorem (A2) it is enough to prove that the function
\beq \label {c-20}
\varphi({\bf u}):\,=\,
\itf h_1(\la)h_2(\la-u_1)h_3(\la-u_2)\cdots h_m(\la-u_{m-1})\,d\la,
\eeq
where ${\bf u}=(u_1,\ldots,u_{m-1})\in\rr^{m-1}$, belongs to
$L^{m-2}({\rr}^{m-1})$ and is continuous at
${\bf 0}=(0,\ldots,0)\in \rr^{m-1}$, provided that
\beq\label{c-21}
h_i\in L^{1}(\rr)\bigcap L^{p_i}(\rr),\quad 1\leq p_i
\leq\infty,\quad i=1,\ldots m,\quad \sum_{i=1}^m \frac1{p_i}\leq 1.
\eeq
It follows from H\"older inequality, (\ref {c-20}) and (\ref{c-21}) that
$$
|\varphi ({\bf u})|\leq\prod_{i=1}^m||f_i||_{L^{p_i}(\rr)}<\infty,
\quad {\bf u}\in \rr^{m-1}.
$$
Hence, $\varphi\in L^\infty(\rr^{m-1})$. On the other
hand, the condition $h_i\in L^{1}(\rr)$ and (\ref{c-20}) imply
$\varphi \in L^1(\rr^{m-1})$. Therefore $\varphi \in L^{m-2}(\rr^{m-1})$.

To prove the continuity of $\varphi ({\bf u})$ at the point
${\bf 0}$ we consider three cases.

{\it Case 1.\/}  $p_i<\infty,\quad i=1,\ldots,m$.

\n For an arbitrary $\varepsilon>0$ we can find $\delta >0$
satisfying (see (\ref{c-21}))
\beq\label{c-22}
\|h_i(\la-u)-h_i(\la)\|_{L^{p_i}(\rr)}\leq\varepsilon,\ \ i=2,\ldots,m,
\quad {\rm if}\quad|u|\leq\delta.
\eeq
We fix ${\bf u}=(u_1,\ldots,u_{m-1})$ with $|{\bf u}| <\delta$ and denote
$$
\overline{h}_i(\la)=h_i(\la-u_{i-1})-h_i(\la),\quad i=2,\ldots,m.
$$
Then in view of (\ref{c-20}) we have
$$
\varphi ({\bf u})=\int_\mathbb{R}h_1(\la)\prod_{i=2}^{m-1}
\left(\overline{h}_i(\la)+h_i(\la)\right)d\la=
\varphi ({\bf 0}) +W.
$$
It follows from (\ref{c-22}) that  $\|\overline{h}_i\|_{p_i}\leq\varepsilon$,
$i=2,\ldots,m$. Observe that each of the integrals comprising $W$
contains at least one function
$\overline{h}_i$ and can be estimated as follows:
\begin{align}
\left|\int_\rr
h_1(u)\overline{h}_2(\la)h_3(\la)\ldots h_{m-1}(\la)
d\la\right|\leq & \notag \\
\|h_1\|_{L^{p_1}}\|\overline{h}_2\|_{L^{p_2}}\|h_3\|_{L^{p_3}}
\ldots \|h_m\|_{L^{p_m}}
\leq & \ C \varepsilon. \notag
\end{align}

{\it Case 2.\/}  $p_i \leq\infty,\ i=1,\ldots,m,\ \ \sum\limits_{i=m}^m \frac1{p_i}<1$.

\n There exist finite numbers $p'_i <p_i,\ i=1,\ldots, m$, such that
$\sum_{i=1}^m 1/{p'_i}\leq 1$. Hence according to (\ref{c-21}) we
have $h_i\in L^{p_i'}\ i=1,\ldots,m$ and $\varphi$ is continuous
at ${\bf 0}$ as in the case 1.

{\it Case 3.\/} $p_i \leq\infty,\ i=1,\ldots,m,\ \sum\limits_{i=1}^m \frac1{p_i}=1$.

\n First observe that at least one of the numbers $p_i$ is finite.
Suppose, without loss of generality, that $p_1<\infty$. For any
$\varepsilon>0$ we can find functions $h_1', h_1''$ such that
\beq\label{c-23}
h_1=h_1'+h_1'',\quad h_1'\in L^\infty(\rr),\quad
\|h_1''\|_{L^{p_1}}<\varepsilon. \eeq
Therefore
$$
\varphi ({\bf u})=\varphi' ({\bf u})+\varphi'' ({\bf u}),
$$
where the functions $\varphi'$ and $\varphi''$  are defined as
$\varphi$ in (\ref {c-20}) with $h_1$ replaced by $h'_1$ and
$h''_1$, respectively. It follows from (\ref{c-23}) that
$\varphi'$ is continuous at ${\bf 0}$ (see Case 2), while
by H\"older inequality $|\varphi''({\bf u})|\leq C\cdot \varepsilon$.
Hence, for sufficiently small $|{\bf u}|$
$$
|\varphi({\bf u})-\varphi({\bf 0})|\leq |\varphi'({\bf
u})-\varphi'({\bf 0})|+ |\varphi''({\bf u})-\varphi''({\bf 0})|\leq
(C+1)\varepsilon,
$$
and the result follows.  Theorem \ref{T5} is proved.
\qed

\vskip2mm
{\sl Proof of Theorem \ref{T5-1}.} We start with {\bf (B1)}.
First observe that the condition $h_i\in {\cal F}_1$ implies that
\beq
\label{hi}
{\hat h}_i\in L^1(\rr)\quad {\rm and}\quad |{\hat h}_i(t)|\leq A,
\quad t\in \rr,\quad i=1,2,\ldots,m
\eeq
for some constant $A>0$.
By Lemma \ref {lem72} we have
$$
T\cdot S(T)=\int_{0}^T\ldots \int_0^T
{\hat h}_1(u_1-u_2){\hat h}_2(u_2-u_3)\ldots {\hat h}_m(u_m-u_1)du_1\ldots du_m.
$$
Making the change of variables
$$
u_1-u_2=t_1,\quad u_2-u_3=t_2, \ldots, u_{m-1}-u_m=t_{m-1},
$$
and observing that $t_1+\cdots+t_{m-1}=u_1-u_m$, we get
(below we use the notation: ${\bf t}_{m-1}= (t_1,\ldots, t_{m-1})$
and $d{\bf t}_{m-1}=dt_1\cdots dt_{m-1}$),
\bea
\label{tst}
T\cdot S(T)=
\int_{0}^T\int_{u_m-t_1-\ldots - t_{m-1}-T}^{u_m-t_1-\ldots - t_{m-1}}
\int_{u_m-t_1-\ldots - t_{m-2}-T}^{u_m-t_1-\ldots - t_{m-2}}\cdots
\eea
\beq
\nonumber
\cdots\int_{u_m-t_1-T}^{u_m-t_1}{\hat h}_1(t_1)
\cdots {\hat h}_{m-1}(t_{m-1}){\hat h}_{m}(-t_1-\ldots-t_{m-1})
d{\bf t}_{m-1}du_m
\eeq
\beq
\nonumber
=\int_{-T}^T\cdots\int_{-T}^T {\hat h}_1(t_1)\cdots
{\hat h}_{m-1}(t_{m-1})
{\hat h}_{m}(-t_1-\ldots-t_{m-1})\left[T-l({\bf t}_{m-1})\right]d{\bf t}_{m-1},
\eeq
where
\beq\label{lt}
\left|l({\bf t}_{m-1})\right|=
\left|l(t_1,\ldots,t_{m-1})\right|\leq2\left(|t_1|+\ldots+|t_{m-1}|
\right).
\eeq
On the other hand, by (\ref{n 4-4}) and Parseval's equality we have
\bea \label{m}
M:&=&M_{{\mathbb R},\mathcal{H}}= (2\pi)^{m-1}\int_{-\f}^{\f}
\left[\prod_{i=1}^m h_i(\la)\right]\,d\la\\ \nonumber
&=&\int_{-\infty}^\infty\cdots\int_{-\infty}^\infty {\hat h}_1(t_1)
\cdots {\hat h}_{m-1}(t_{m-1}){\hat h}_{m}(-t_1-\ldots-t_{m-1})
d{\bf t}_{m-1}.
\eea
It follows from (\ref{n 4-4}), (\ref{tst}) and (\ref {m}) that
\beq
\nonumber
S(T)-M:=S_{W,\cal H}(T)-M_{{\mathbb R},\mathcal{H}}
\eeq
\beq
\nonumber
= -\frac1T\int_{[-T,\,T]^{m-1}} {\hat h}_1(t_1)
\cdots {\hat h}_{m-1}(t_{m-1}){\hat h}_{m}(-t_1-\ldots-t_{m-1})l({\bf t}_{m-1})
d{\bf t}_{m-1}
\eeq
\bea
\label{delta}
\nonumber
&&+
\int_{\rr^{m-1}\setminus[-T,\,T]^{m-1}} {\hat h}_1(t_1)
\cdots {\hat h}_{m-1}(t_{m-1}){\hat h}_{m}(-t_1-\ldots-t_{m-1})d{\bf t}_{m-1}\\
&&=: \Delta_T^1+\Delta_T^2.
\eea
By (\ref{hi}), (\ref{lt}) and (\ref{delta}) we have
\beq \label{delta1}
|T\cdot\Delta_T^1|\leq 2A
\sum_{i=1}^{m-1}\int_{\rr^{m-1}} \big|{\hat h}_1(t_1)
\cdots {\hat h}_{m-1}(t_{m-1})t_i\big|
d{\bf t}_{m-1}=:A_1<\infty,
\eeq
since $h_i\in {\cal F}_1$, $i=1,2,\ldots,m$.

Further, observe that
$$
{\rr^{m-1}\setminus[-T,\,T]^{m-1}}\subset\bigcup_{i=1}^m
\big\{(t_1,\ldots,t_{m-1})\in\rr^{m-1}:|t_i|>T\big\}=:\bigcup_{i=1}^m E_i.
$$
Hence by (\ref{hi}) and (\ref{delta}) we have
\beq \label{delta2}
|T\cdot\Delta_T^2|\leq 2A
\sum_{i=1}^{m-1}\int_{E_i}\big| {\hat h}_1(t_1)
\ldots {\hat h}_{m-1}(t_{m-1})t_i\big|
d{\bf t}_{m-1}=:A_2<\infty.
\eeq
From (\ref{delta}) -- (\ref{delta2}) we get {\bf (B1)}.
\qed

\vskip2mm
\n
{\sl Proof of {\bf (B2)}}. By Lemma \ref{lem72}, 3) and
Lemma \ref{lem73}, b), and (\ref{n 4-5}) we have
\beq \label{a-3}
\Delta(T)=(2\pi)^{m-1}
\left|\int_{\rr^{m-1}}[\Psi({\bf u})-\Psi({\bf 0})]
\Phi_T({\bf u})d{\bf u}\right|,\quad {\bf 0}=(0,\ldots,0).
\eeq
It follows from (\ref{n 5-3}) and (\ref{lem72-2}) that
for ${\bf u}=(u_1, \ldots, u_{m-1})\in \rr^{m-1}$
\beq
\label{w2}
|\Psi({\bf u})-\Psi({\bf 0})|\leq (m-1)C
\left(|u_1|^\g +\cdots+|u_{2\nu-1}|^\g\right).
\eeq

\n Let $\varepsilon\in (0,\g)$. Then, applying
Lemma \ref{lem01} with $\de=\frac{1+\varepsilon-\g}{m}$,
and using (\ref{a-3}) and (\ref{w2}), we can write
\bea
\label{4-5} \nonumber
&&\Delta(T) \leq  C_m
\int_{E}|\Psi({\bf u})-\Psi({\bf 0})||\Phi_T({\bf u})|d{\bf u}\\
&&+C_m \int_{E^C}|\Psi({\bf u})-\Psi({\bf 0})||\Phi_T({\bf u})|d{\bf u}\\
\nonumber
&&\leq \frac {C_m}{T^{1-m\delta}} \sum\limits _{i=1}^{m-1}
 \int_{E}
\frac{|u_i|^\g}{|u_1\cdots u_{m-1}(u_1+\cdots+u_{m-1})|^{1-\delta}} \,
du_1\cdots du_{m-1}\\ \nonumber
&&+ 2\|\varphi\|_\infty \frac{C_m}{T^{1-m\delta}}\int_{E^C}
\frac{1}{|u_1\cdots u_{m-1}(u_1+\cdots+u_{m-1})|^{1-\delta}}
\, du_1\cdots du_{m-1},
\eea
where $E=\{(u_1,u_2,\ldots,u_{m-1})\in\rr^{m-1}: |u_i|\leq1,
\ i=1,2,\ldots,{m-1}\}$ and $E^C=\rr^{m-1}\setminus E$. Since
$$
\frac{m-1}{m}<1-\de<\frac{m-1+\g}{m},
$$
we can apply Lemmas \ref{lem2} and \ref{lem3}
with $\alpha=\gamma$, $n=m-1$ and $\beta=1-\delta$
to conclude that all the integrals in (\ref{4-5}) are finite.
Since $1-m\de=\g-\varepsilon$, from
(\ref{4-5}) follows the statement  {\bf (B2)}.
\qed

\vskip2mm
\n
{\sl  Proof of {\bf (B3)}}.
According to {\bf (B2)} it is enough to prove that
the function
\beq \label {c-2-0}
\varphi ({\bf u}):=\int_\rr
h_1(\la)h_2(\la-u_1)\cdots h_{m}(\la-u_{m-1})d{\bf u},
\eeq
where ${\bf u}=(u_1,\cdots,u_{m-1})\in\rr^{m-1},$ belongs to
$L^\infty(\rr^{m-1})$, and with some positive constant $C$
\beq\label{lip1}
|\varphi ({\bf u})-\varphi
({\bf 0})|\leq C|{\bf u}|^\g,
\q {\bf u}\in \rr^{m-1},
\eeq
provided that
\beq\label{c-2-1}
h_i\in \Lip(\rr;p_i, \g), \quad 1\leq p_i \leq\infty,\quad
i=1,2,\ldots,\,{m},\q {\rm and} \q  \sum_{i=1}^{m} \frac1{p_i}\leq 1.
\eeq
It follows
from H\"older inequality and (\ref{c-2-1}) that
$$
|\varphi ({\bf u})|\leq\prod_{i=1}^{m}||h_i||_{{p_i}}<\infty,
\quad {\bf u}\in \rr^{m-1}.
$$
Hence $\varphi\in L^\infty(\rr^{m-1})$.

To prove (\ref{lip1}) we fix ${\bf u}=(u_1,\ldots,u_{m-1})\in\rr^{m-1}$
and denote
\beq
\label{c-2-2}
\overline{h}_i(\la)=h_i(\la-u_{i-1})-h_i(\la),\quad \la\in \mathbb{R},
\quad i=2,\ldots,{m}.
\eeq
Since $h_i\in \Lip(\rr;p_i, \g)$ we have
\beq\label{c-2-3}
\|\overline{h}_i\|_{p_i}\leq C_i|{\bf u}|^\g,\q i=2,\ldots,{m}.
\eeq
By (\ref{c-2-0}) and (\ref{c-2-2}),
$$
\varphi ({\bf u})=\int_\mathbb{R}h_1(\la)\prod_{i=2}^{m}
\left(\overline{h}_i(\la)+h_i(\la)\right)d\la=
\varphi ({\bf 0}) + W.
$$
Each of the $(2^{m-1}-1)$ integrals comprising $W$ contains
at least one function $\overline{f}_i$, and in view
of (\ref{c-2-3}), can be estimated as follows:
$$
\left|\int_\rr
h_1(\la)\overline{h}_2(\la)h_3(\la)\cdots h_{m}({\la})d\la\right|
$$
$$
\leq \|h_1\|_{{p_1}}\|\overline{h}_2\|_{{p_2}}\|h_3(u)\|_{{p_2}}
\cdots \|h_{m}\|_{{p_{m}}}\leq C|{\bf u}|^\g.
$$
This completes the proof of {\bf (B3)}.
\qed

\vskip2mm
\n
{\sl Proof of {\bf (B4)}.}
We set
$$
\frac 1{p_i}:=\sigma_i+\frac1{m}\left[1-(\sigma_1+\cdots\sigma_m)\right],
\quad i=1,2,\ldots,m.
$$
Then
$$
\sum_{i=1}^m\frac1{p_i}=1 \q{\rm and} \q
\frac 1{p_i}-\sigma_i+\frac1{m}\left[1-(\sigma_1+\cdots\sigma_m)\right]=\gamma>0,
\quad
i=1,2,\ldots,m,
$$
and
$$
0<\sigma_i<\frac1p_{i}<1<\delta_i,\qq i=1,2,\ldots,m.
$$
Hence, according to Lemma \ref{R1}, $h_i\in\Lip(p_i, \g)$,
$i=1,2,\ldots,m$.
Applying (B3), we get
(\ref{n 5-6}) with $\gamma$ as in (\ref{n 5-7}).
\qed

\vskip2mm
\n {\it Proof of Theorem \ref{T5-2}.\/}
First observe that by Lemma 7 from \cite{GS5}
\beq
\label{F15}
S_{2, W}(T):=\frac1T\tr[W_T(h_1)W_T(h_2)]= 2\pi\int_\mathbb{R}\int_\mathbb{R} F_T(s-t)h_1(s)h_2(t)\,dt\,ds,
\eeq
where $F_T(u)$ is the  Fej\'er kernel:
$$ F_T(u)=\frac 1{2\pi T} \left (\frac {\sin{{Tu}/2}}{u/2}\right )^2,
\q t \in \mathbb{R}.
$$
Below we use the following properties of $F_T(u)$ (see, e.g., \cite{BN}):
\bea
\label{F12}
&&\int_\mathbb{R} F_T(u)\,du = 1,\\
\label{F13}
&&\int_{u \ge 1}F_T(u)du\le C\,T^{-1},\\
\label{F14}
&&
\int_0^1F_T(u)u^\al du\le \left \{
           \begin{array}{lll}
           CT^{-\al}, & \mbox {if \, $\al\le 1$}\\
           CT^{-1}\ln T, & \mbox {if \, $\al=1$}\\
           CT^{-1}, & \mbox {if \, $\al>1$}.
           \end{array}
           \right.
\eea
Since the function $F_T(u)$ is even, in view of (\ref{F15}) we can write
\bea
\label{F17}
S_{2, W}(T)=\pi\int_\mathbb{R}\int_\mathbb{R}
F_T(u)\left[h_1(u+t)h_2(t) + h_1(t)h_2(u+t)\right]dudt.
\eea
Consequently, taking into account (\ref{F12}) and the equality
$$\int_\mathbb{R} h_1(t)h_2(t)dt=\int_\mathbb{R} h_1(u+t)h_2(u+t)dt,$$
by (\ref{F17}) we get
\bea
\De_{2, W}(T):&=&
\left|\frac1T\tr[W_T(h_1)W_T(h_2)]- 2\pi\int_{\mathbb{R}}h_1(t)h_2(t)\,dt\right|\\
\label{F18}
\nonumber
&=&\left|\pi\int_\mathbb{R} F_T(u)
\int_\mathbb{R}(h_1(t)-h_1(u+t))(h_2(u+t)-h_2(t))dtdu\right|.
\eea
Using H\"older's inequality, we find from (\ref{F18})
\bea
\label{F19}
\De_{2, W}(T)\le\pi\int_\mathbb{R}
F_T(u)||h_1(u+\cdot)-h_1(\cdot)||_p||h_2(u+\cdot)-h_2(\cdot)||_qdu.
\eea
In view of (\ref{F19}) we have
\bea
\label{F20}
\De_{2, W}(T) \le C_1\int_0^1 F_T(u)|u|^{\g_1+\g_2}du +
C_2||h_1||_p||h_2||_q\int_{u>1}F_T(u)du.
\eea
Therefore, the result follows from (\ref{F13}), (\ref{F14}) and (\ref{F20}).
\qed

\vskip2mm
\n {\it Proof of Theorem \ref{T5-3}.\/}
It follows from Lemma \ref{R1} that under the assumptions of theorem
$h_i\in \Lip(p_i, 1/{p_i}-\sigma_i)$, $i=1,2$.
Hence, applying Theorem  \ref{T5-2} with
$\g_i= 1/{p_i}-\sigma_i$, we obtain (\ref{t53}).
\qed

\s{Appendix. Proof of Technical Lemmas}
\label{TL}

\n
In this section we give proofs of technical lemmas stated
and used in Sections \ref{ASP} and \ref{PR}.

\vskip2mm
\n {\it Proof of Lemma \ref{R1}.\/}
Let $h\in (0,1/2)$ be fixed. Then
\begin{equation}\label{R4-1}
\int\limits_{|\la|\le2h}|f(\la+h)-f(\la)|^pd\la
\le (2C)^p \int_0^{3h} |\la|^{-p\sigma} d\la \le C_1 h^{1-p\alpha}.
\end{equation}
Next, for $|\la|>2h$ we have
$f(\la+h)-f(\la)=f^\prime(\xi)\cdot h$ with some $\xi\in(\la,\la+h)$.
Hence
\begin{equation}\label{R4-3}
\int\limits_{2h<|\la|<1/2}|f(\la+h)-f(\la)|^pd\la
\le C^p h^p\int_h^{1/2} |\la|^{-p(\sigma+1)} d\la \le C h^{1-p\sigma}
\end{equation}
and
\bea\label{R4-2}
\nonumber
\int\limits_{1/2<|\la|<\f}|f(\la+h)-f(\la)|^pd\la
&&\le C^p h^p \int_{1/2}^1 |\la|^{-p(\sigma+1)}d\la \\ \nonumber
&&+C^p h^p \int_1^\infty |\la|^{-p(\delta+1)}d\la\\
&&\le C h^p\le C h^{1-p\sigma}.
\eea
From (\ref{R4-1}), (\ref{R4-3}) and (\ref{R4-2}) we get
$$
\|f(\la+h)-f(\la)\|_p\le Ch^{1/p-\sigma},
$$
implying $f\in \Lip(p, 1/p-\sigma)$.
\qed

\vskip2mm
\n {\it Proof of Lemma \ref{sl1}.\/}
We use the technique of \cite{IK1}, where
(\ref{s4}) was proved for $\be=1$.
Since the underlying process $X(t)$ is real-valued, we have for $t>0$
\beq
\label{sa1}
r(t):=\itf e^{it\la}f(\la)\,d\la=
2\int_0^\f\frac C{\la^{2\al}(1+\la^2)^\be}\cos(t\la)\,d\la.
\eeq
Using the change of variable $\la=1/(tu)$, we obtain
\bea
\label{sa02}
\nonumber
r(t) &=&2C\cd t^{2\al-1}\int_0^\f\left(\frac {(tu)^2}{1+(tu)^2}\right)^\be
u^{2\al-2}\cos(1/u)\,du\\
&=&2C\cd t^{2\al-1}\int_0^\f L(tu)k(u)\,du,
\eea
where
$$
%L(u)=\left(\frac {u^2}{1+u^2}\right)^\be
L(u)=\frac {u^{2\be}}{(1+u^2)^\be}
\q {\rm and}\q
k(u)=u^{2\al-2}\cos(1/u).
$$

\n Choose $\de>0$ such that $\de<\min(1-2\al, 2\al)$.
Then the improper integrals
$$
\int_{0+}^1 u^{-\de}k(u)\,du \q {\rm and}\q
\int_1^{\f-} u^{\de}k(u)\,du
$$
exist. Therefore, by the Bojanic-Karamata theorem
(see \cite{BGT}, Th.4.1.5 ) we have
\beq
\label{sa2}
\int_0^\f k(u)L(tu)\,du\longrightarrow \int_{0+}^{\f-} k(u)\,du
\q {\rm as}\q t\to\f.
\eeq
Next, using the change of variable $1/u=v$ and the formula
(see, e.g., \cite{Dv})
$$
\int_{0}^{\f-} x^{-p}\cos(mx)\,dx
=\frac{\pi m^{p-1}}{2\cos(p\pi/2)\G(p)}, \q 0<p<1, \q m>0,
$$
with $p=2\al$ and $m=1$, we obtain
\bea
\label{sa3}
\int_{0+}^{\f-} k(u)\,du &=&\int_{0}^{\f-} u^{2\al-2}\cos(1/u)\,du
=\frac{\pi}{2\cos(\pi\al)\G(2\al)}.
\eea
From (\ref{sa1})--(\ref{sa3}) we get (\ref{s4}).
\qed

\vskip2mm
\n {\it Proof of Lemma \ref{lem73}.\/}
The proof of properties a) - c) can be found in  \cite {B}, Lemma 3.2
(see also \cite {GS2}, Lemma 2).
To prove d) first observe that for $T>0$
\beq \label{z1}
\int_{\mathbb{R}}\left|D_T(u)\right|^{p(m)}du\leq C\cdot T^{p(m)-1}
\quad \mbox{and} \quad
\left|D_T(u)\right|\leq C_\de  \quad \mbox{for} \quad |u|>\de.
\eeq
For $\uu=(u_1,\ldots,u_{m-1})\in {\mathbb R }^{m-1}$ we have
\bea \label{z2}
\nonumber
\int_{\mathbb{E}^c_\de}\Phi_T^{p(m)}(\uu)d\uu&\le&
\int\limits_{|u_1|>\de}\Phi_T^{p(m)}(\uu)d\uu
+ \int\limits_{|u_2|>\de}\Phi_T^{p(m)}(\uu)d\uu \\ \nonumber
&+&\ldots+\int\limits_{|u_{m-1}|>\de}\Phi_T^{p(m)}(\uu)d\uu\\
&=:& I_1+I_2+\ldots+I_{m-1}.
\eea
It is enough to estimate $I_1$ ($I_2, \ldots, I_n$ can be estimated in the same way). We have
\bea \label{z3}
\nonumber
I_1&\le& \int\limits_{|u_1|>\de, \,\,
|u_2|>\de/m}\Phi_T^{p(m)}(\uu)d\uu
+ \ldots+\int\limits_{|u_1|>\de, \,\, |u_{m-1}|>\de/m}\Phi_T^{p(m)}(\uu)d\uu \\
\nonumber
&+&\int\limits_{|u_1|>\de, \,\, |u_2|\leq\de/m, \,\,
|u_{m-1}|\leq\de/m}\Phi_T^{p(m)}(\uu)d\uu\\
&=:& I_1^{(2)}+\ldots+I_1^{(m-1)}+I_1^{(m)}.
\eea
According to (\ref{z1})
\begin{align}\label{z4}
\nonumber
I_1^{(2)}&\leq C_\de \cdot \frac1{T^{p(m)}}\cdot
\int\limits_{|u_2|>\de/m}\left|D_T(u_2)\right|^{p(m)}
\cdots \left|D_T(u_{m-1})\right|^{p(m)}\\
&\times
\left|D_T(u_1+\ldots+u_{m-1})\right|^{p(m)}du_1du_{m-1}\ldots du_2\notag\\
&\le C_\de\cdot
\frac1{T^{p(m)}}\cdot T^{(p(m)-1)(m-2)}
\int\limits_{|u_2|>\de/m}\frac1{|u_2|^{p(m)}} \, du_2\le C_\de.
\end{align}
Likewise,
\beq \label{z5}
I_1^{(j)}\leq C_\de,\qquad j=3,\ldots,  m-1.
\eeq
Next, observe that in the integral $I_1^{(m)}$, we have
$|u_1+\ldots+u_{m-1}|>\de/m,$. Hence by (\ref{z1})
\begin{align}\label{z6}
I_1^{(m)}&\leq C_{\de} \cd \frac1{T^{p(m)}}\int\limits_{|u_1|>\de}D^{p(m)}_T(u_1)\cdots
D^{p(m)}_n(u_{m-1})du_2\ldots du_{m-1}du_1\notag\\
&\le C_\de\int\limits_{|u_1|>\de}\frac1{|u_1|^{p(m)}} \, du_1\le C_\de.
\end{align}
From (\ref{z2}) -- (\ref{z6}) we obtain (\ref{z0}).
Lemma \ref{lem73} is proved.

\vskip2mm
\n {\it Proof of Lemma \ref{lem2}.\/}
Using Lemma \ref{lem1} and the notation
$d{\bf u}=du_ndu_{n-1}\cdots du_1$, we can write
$$
B_1\leq\int\limits_{\{|u_1|\leq1\}}\frac1{|u_1|^{\beta-\alpha}}
\int\limits_{\rr^{n-2}}\frac1{|u_2\cdots u_{n-1}|^\beta}
\int\limits_{\rr}\frac1{|u_n(u_1+\cdots+u_n)|^\beta} \,d{\bf u}
%du_ndu_{n-1}\cdots du_1
$$
$$
\leq C\int\limits_{\{|u_1|\leq1\}}\frac1{|u_1|^{\beta-\alpha}}
\int\limits_{\rr^{n-2}}\frac1{|u_2\cdots u_{n-1}|^\beta|(u_1+\cdots+
 u_{n-1})|^{2\beta-1}}du_{n-1}\cdots du_1
$$
$$
\leq C^2\int\limits_{\{|u_1|\leq1\}}\frac1{|u_1|^{\beta-\alpha}}
\int\limits_{\rr^{n-3}}\frac1{|u_2\cdots u_{n-2}|^\beta|(u_1+\cdots+
 u_{n-2})|^{3\beta-2}}du_{n-2}\cdots du_1
 $$
\beaa
&\leq&
\ldots\ldots\ldots\ldots\ldots\ldots\ldots\ldots\ldots\ldots\ldots
\ldots\ldots\ldots\ldots\ldots\ldots\ldots\\
&\leq&
C^{n-2}\int\limits_{\{|u_1|\leq1\}}\frac1{|u_1|^{\beta-\alpha}}
\int\limits_{\rr}\frac1{|u_2|^\beta|(u_1+ u_2)|^{(n-1)\beta-n+2}}du_2du_1\\
&\leq&
C^{n-1}\int\limits_{\{|u_1|\leq1\}}\frac1{|u_1|^{(n+1)\beta-\alpha-n+1}}
du_1<\infty,
\eeaa
yielding (\ref{a-05}) for $i=1$.
The quantities $B_2,\ldots, B_n$ can be estimated in the same way.

\vskip2mm

\n {\it Proof of Lemma \ref{lem3}.\/}
We have
\begin{equation}
\label{a-76}
I\leq
\int\limits_{|u_1|>1}+\cdots+\int\limits_{|u_n|>1}
\frac1{|u_1\cdots u_n(u_1+\cdots+u_n)|^\beta} \, du_1\cdots du_n
=:I_1+\cdots+I_n.
\eeq
Using Lemma \ref{lem1} we get
\beaa
I_1&\leq&
\int\limits_{|u_1|>1}\frac1{|u_1|^\beta}\int\limits_{\rr^{n-2}}
\frac1{|u_2\cdots u_{n-1}|^\beta}
\int\limits_{\rr}\frac1{|u_n(u_1+\cdots+u_n)|^\beta} \,
du_n\cdots du_1\\
&\leq&
C\int\limits_{|u_1|>1}\frac1{|u_1|^\beta}\int\limits_{\rr^{n-2}}
\frac1{|u_2\cdots u_{n-1}|^\beta |u_1+\cdots+u_{n-1}|^{2\beta-1}} \,
du_{n-1}\cdots du_1\\
&\leq&\ldots\ldots\ldots\ldots\ldots\ldots\ldots\ldots\ldots\ldots\ldots
\ldots\ldots\ldots\ldots\ldots\ldots\ldots\ldots\ldots\ldots\\
&\leq&
C^{n-1}\int\limits_{|u_1|>1}\frac1{|u_1|^{(n+1)\beta-n+1}} \, du_1<\infty.
\eeaa
The quantities $I_2, \ldots, I_n$ can be estimated in the same way,
and by (\ref{a-76}) the result follows.

\end{document}